\documentclass{amsart}
\usepackage{amscd,amssymb,amsopn,amsmath,amsthm,graphics,amsfonts,accents,enumerate,verbatim,calc}
\usepackage[dvips]{graphicx}
\usepackage{pgfplots}
\usepackage[colorlinks=true,linkcolor=red,citecolor=blue]{hyperref}
\usepackage[all]{xy}
\usepackage{cite}
\usepackage{tikz}
\usepackage{tikz-cd}
\usepackage{mathrsfs}

\calclayout

\newcommand{\rt}{\rightarrow}
\newcommand{\lrt}{\longrightarrow}

\newcommand{\st}{\stackrel}

\newcommand{\la}{\lambda}
\newcommand{\La}{\Lambda}

\newcommand{\SA}{\mathscr{A}}
\newcommand{\SB}{\mathscr{B}}
\newcommand{\SC}{\mathscr{C}}
\newcommand{\SD}{\mathscr{D}}

\newcommand{\SF}{\mathscr{F}}

\newcommand{\SR}{\mathscr{R}}

\newcommand{\ST}{\mathscr{T}}
\newcommand{\SU}{\mathscr{U}}
\newcommand{\SX}{\mathscr{X}}
\newcommand{\SY}{\mathscr{Y}}

\newcommand{\CC}{\mathcal{C} }

\newcommand{\CX}{\mathcal{X} }

\newcommand{\Mod}{{\rm{Mod\mbox{-}}}}

\newcommand{\mmod}{{\rm{{mod\mbox{-}}}}}

\newcommand{\Inj}{{\rm{Inj}}}
\newcommand{\Prj}{{\rm{Prj}}}

\newcommand{\Prod}{{\rm{Prod}}}

\newcommand{\im}{{\rm{Im}}}

\newcommand{\add}{{\rm{add}}}

\newcommand{\Add}{{\rm{Add}}}

\newcommand{\Fac}{{\rm{Fac}}}
\newcommand{\Sub}{{\rm{Sub}}}
\newcommand{\ann}{{\rm{ann}}}

\newcommand{\Pres}{{\rm Pres}}
\newcommand{\Copres}{{\rm Copres}}

\newcommand{\Coker}{{\rm{Coker}}}
\newcommand{\Ker}{{\rm{Ker}}}

\newcommand{\Rep}{{\rm Rep}}

\newcommand{\RepR}{{\rm Rep}^{\rm{fp}}_k \mathbb{R}_{\geq 0}}
\newcommand{\RepRO}{{\rm Rep}^{\rm{fp}}_k\mathbb{R}_{\geq 1}}

\newcommand{\Hom}{{\rm{Hom}}}
\newcommand{\Ext}{{\rm{Ext}}}

\newcommand{\Gen}{{\rm{Gen}}}
\newcommand{\Cogen}{{\rm{Cogen}}}

\theoremstyle{plain}
\newtheorem{theorem}{Theorem}[section]
\newtheorem{corollary}[theorem]{Corollary}
\newtheorem{lemma}[theorem]{Lemma}

\newtheorem{proposition}[theorem]{Proposition}

\theoremstyle{definition}
\newtheorem{definition}[theorem]{Definition}
\newtheorem{example}[theorem]{Example}

\newtheorem{remark}[theorem]{Remark}

\newtheorem{s}[theorem]{}

\theoremstyle{plain}
\newtheorem{stheorem}{Theorem}[subsection]
\newtheorem{scorollary}[stheorem]{Corollary}
\newtheorem{slemma}[stheorem]{Lemma}
\newtheorem{sproposition}[stheorem]{Proposition}

\theoremstyle{definition}
\newtheorem{sdefinition}[stheorem]{Definition}
\newtheorem{sexample}[stheorem]{Example}

\newtheorem{sremark}[stheorem]{Remark}

\numberwithin{equation}{section}

\newcommand{\Dim}[1]{%
  {%
    \tiny%
    \begin{matrix}%
      #1%
    \end{matrix}%
  }%
}

\begin{document}

\title[($\tau$-)tilting subcategories]{On $\tau$-tilting subcategories}

\author[J. Asadollahi, S. Sadeghi and H. Treffinger]{Javad Asadollahi, Somayeh Sadeghi and Hipolito Treffinger}

\address{Department of Pure Mathematics, Faculty of Mathematics and Statistics, University of Isfahan, P.O.Box: 81746-73441, Isfahan, Iran}
\email{asadollahi@sci.ui.ac.ir, asadollahi@ipm.ir }
\email{so.sadeghi@sci.ui.ac.ir }

\address{Universit\'{e} de Paris, B\^{a}timent Sophie Germain 5, rue Thomas Mann 75205, Paris Cedex 13, FRANCE}
\email{treffinger@imj-prg.fr}

\makeatletter \@namedef{subjclassname@2020}{\textup{2020} Mathematics Subject Classification} \makeatother



\begin{abstract}
The main theme of this paper is to study $\tau$-tilting subcategories in an abelian category $\SA$ with enough projective objects.
We introduce the notion of $\tau$-cotorsion torsion triples and show a bijection between the collection of $\tau$-cotorsion torsion triples in $\SA$ and the collection of $\tau$-tilting subcategories of $\SA$, generalizing the bijection by Bauer, Botnan, Oppermann and Steen between the collection of cotorsion torsion triples and the collection of tilting subcategories of $\SA$. General definitions and results are exemplified using persistent modules.
If $\SA=\Mod R$, where $R$ is an unitary associative ring, we characterize all support $\tau$-tilting, resp. all support $\tau^-$-tilting, subcategories of $\Mod R$ in term of finendo quasitilting, resp. quasicotilting, modules. 
As a result, it will be shown that every silting module, respectively every cosilting module, induces a support $\tau$-tilting, respectively support $\tau^{-}$-tilting, subcategory of $\Mod R$. 
We also study the theory in $\Rep(Q, \SA)$, where $Q$ is a finite and acyclic quiver. 
In particular, we give an algorithm to construct support $\tau$-tilting subcategories in $\Rep(Q, \SA)$ from certain support $\tau$-tilting subcategories of $\SA$ and present a systematic way to construct $(n+1)$-tilting subcategories in $\Rep(Q, \SA)$ from $n$-tilting subcategories in $\SA$.
\end{abstract}

\maketitle

\tableofcontents

\section{Introduction}
Tilting theory is one of the most prominent tools in representation theory of artin algebras. The classical tilting modules were introduced by Brenner and Butler \cite{BB}, and Happel and Ringel \cite{HR} as an axiomatisation of the reflection functors of Bernstein, Gelfand and Ponomarev \cite{BGP} and Auslander, Platzek and Reiten \cite{APR}.
It has been shown by Bongartz \cite{Bon} that every partial tilting module can be completed to a tilting module and by Happel and Unger \cite{HappelUnger} that every almost complete tilting module can be completed in at most two ways.
However, there are examples of almost complete tilting modules that have exactly one complement.

Several years later, and inspired by the cluster algebras defined by Fomin and Zelevnisky in \cite{FZ1}, Adachi, Iyama and Reiten introduced $\tau$-tilting theory \cite{AIR}.
This is a generalization of classical tilting theory in which every almost complete support $\tau$-tilting module has exactly two complements, allowing to introduce a notion of mutation among these objects.
The success of $\tau$-tilting theory was immediate, offering an explanation for several phenomena in the module category of artin algebras and offering new connections between representation theory and other areas of mathematics (see \cite{TreffingerSurvey}).

Due to the effectiveness of $\tau$-tilting theory for the study of the categories of finitely presented modules, many mathematicians have introduced theories generalizing $\tau$-tilting theory, and its dual, to other contexts.
For instance, there are the works of Angeleri-H\"ugel, Marks and Vitoria \cite{AMV} and Breaz and Pop \cite{BP} for the module category of rings; Iyama, J\o rgensen and Yang \cite{IJY} for functor categories; or Liu and Zhou for Hom-finite abelian categories with enough projective objects \cite{LZh}.

In this paper we are interested in studying $\tau$-tilting theory in arbitrary abelian categories with enough projective objects.
Since in general there is no notion of Auslander-Reiten translation $\tau$ in such general categories, the definition of support $\tau$-tilting subcategories needs to be made with no mention to it. We follow \cite{IJY, LZh} and define support $\tau$-tilting subcategories as follows.

\begin{definition}
Let $\SA$ be an abelian category with enough projective objects.
Let $\ST$ be an additive  contravariantly finite full subcategory of $\SA$.
Then $\ST$ is called a support $\tau$-tilting subcategory if
\begin{itemize}
  \item[$1.$] $\Ext^1_{\SA}(T_1,\Fac(T_2))=0,$ for all $T_1, T_2 \in \ST$.
  \item[$2.$] For any projective $P$ in $\SA$, there exists an exact sequence
  \[P \st{f}{\lrt} T^0 \lrt T^1 \lrt 0\]
  such that $T^0$ and $T^1$ are in $\ST$ and $f$ is a left $\ST$-approximation of $P$.
\end{itemize}
\end{definition}

The idea of generalizing a well-behaved theory from the category of finitely presented modules of an artin algebras to more general abelian categories is not exclusive to $\tau$-tilting theory.
Indeed, Beligiannis introduced tilting theory for arbitrary abelian categories in \cite{B} at the beginning of the millenium, see also \cite{BR}.
In recent years, the work of Bauer, Botnan, Oppermann and Steen \cite{BBOS} has found a striking application of this theory in persistence theory and topological data analysis (TDA).
One of the main results in \cite{BBOS} is a bijection between tilting subcategories of an abelian category with enough projective objects and the collection of cotorsion torsion triples in the same category.
The motivation that started this collaboration is to find a generalization of \cite[Theorem~2.29]{BBOS} to support $\tau$-tilting subcategories, but for that we first need an adequate notion of triple.
Inspired by \cite[Lemma~V.3.3]{BR} we give the definition of $\tau$-cotorsion torsion triple as follows.

\begin{definition}
Let $\SA$ be an abelian category with enough projective objects.
A triple of full subcategories $(\SC, \SD, \SF)$ of $\SA$ is called a $\tau$-cotorsion torsion triple if
\begin{itemize}
\item[$1.$] $\SC = {}^{\perp_1}\SD$.
\item[$2.$] For every projective object $P \in \SA$, there exists an exact sequence
\[  P \st{f}\lrt {D}\lrt C \lrt 0,\]
 where $D \in \SC\cap\SD$, $C \in \SC$ and $f$ is a left $\SD$-approximation.
\item[$3.$] $\SC\cap\SD$ is a contravariantly finite subcategory of $\SA$.
\item[$4.$] $(\SD, \SF)$ is a torsion pair in $\SA$.
\end{itemize}
\end{definition}

Having the notion of $\tau$-cotorsion torsion triple we are able to show the desired bijection in an abelian category with enough projective objects.

\begin{theorem}[Theorem~\ref{Bijection}]
Let $\SA$ be an abelian category with enough projective  objects.
Then there are bijections
\begin{align*}
\Phi:  \{\mbox{support $\tau$-tilting subcategories}\}& \longrightarrow \{ \mbox{$\tau$-cotorsion torsion triples}\}\\
\ST& \longmapsto ({}^{\perp_1}{(\rm{Fac}(\ST))}, \rm{Fac}(\ST), \ST^{\perp_0})\\
\\
\Psi:  \{ \mbox{$\tau$-cotorsion torsion triples}\} & \longrightarrow \{\mbox{support $\tau$-tilting subcategories}\}\\
(\SC, \SD, \SF) & \longmapsto \SC\cap\SD
\end{align*}
which are mutually inverse. Moreover these bijections restrict to bijections between the collection of tilting subcategories of $\SA$ and the collection of cotorsion torsion triples in $\SA$.
\end{theorem}

We note that Buan and Zhou \cite{BZ} introduced the notion of left weak cotorsion torsion triple in order to provide a version of  \cite[Theorem~2.29]{BBOS} for support $\tau$-tilting modules in $\mmod\La$, where $\La$ is an artin algebra. In Theorem~\ref{lwcotorsion} we show that a triple $(\SC, \SD, \SF)$ of subcategories of $\mmod\La$ is a left weak cotorsion torsion triple if and only if it is a $\tau$-cotorsion torsion triple.

One of the main properties of support $\tau$-tilting modules is that they generate functorially finite torsion classes in module categories.
As a consequence of our results, we obtain the following corollary, see Theorem~\ref{Bijection} and Corollary~\ref{Cor-ff}.

\begin{corollary}
Let $\SA$ be an abelian category with enough projective objects and let $\ST$ be a support $\tau$-tilting subcategory of $\SA$.
Then $\Fac(\ST)$ is a functorially finite torsion class in $\SA$.
\end{corollary}

For $\mmod\La$, using results of \cite{AIR}, it is shown that a module $T$ in $\mmod\La$ is a support $\tau$-tilting module if and only if $\add(T)$ is a support $\tau$-tilting subcategory of $\mmod \La$. On the other hand, $T$ is a support $\tau^-$-tilting module if and only if $\add(T)$ is a support $\tau^-$-tilting subcategory of $\mmod \La$, see Propositions~\ref{EquivalentDefinitions} and \ref{DualEquivalentDefinitions}.

Suppose that $\SA$ is an abelian category with enough injective objects.
Then in this category all the dual definitions and results can be stated and proved, see Section~\ref{Sec:Dual}.
Using this symmetry, when $\SA$ is an abelian category with enough projective and enough injective objects, we introduce the notion of a $\tau$-$\tau^-$-quadruple and use it to create a map from support $\tau$-tilting to support $\tau^-$-tilting subcategories of $\SA$.
We show in Proposition~\ref{daggermap} that this map is exactly the dual of the dagger map defined in \cite[Theorem~2.15]{AIR} when $\SA=\mmod\La$.

\medskip
Besides $\mmod \La$, another prominent example of an abelian category with enough projective objects is $\Mod R$, the category of modules over an unitary associative ring $R$.

The theory of support $\tau$-tilting modules of \cite{AIR} is generalized to $\Mod R$ by Angeleri H\"{u}gel, Marks and Vitoria \cite{AMV}, where they introduced silting modules showing that finitely presented silting $\La$-modules coincide with the support $\tau$-tilting modules of \cite{AIR}. As the categorical dual of silting modules, Breaz and Pop introduced the notion of cosilting modules in $\Mod R$ and showed that finitely copresented cosilting $\La$-modules coincide with the support $\tau^-$-tilting modules.

We study the connection between silting, resp. cosilting, modules in $\Mod R$ with the $\tau$-tilting, resp. $\tau^-$-tilting subcategories of $\Mod R$. Furthermore, we characterize all support $\tau$-tilting subcategories of $\Mod R$. 
In particular, we provide a bijection between the equivalence classes of all support $\tau$-tilting subcategories of $\Mod R$ and the collection of all equivalent classes of certain $R$-modules, the so-called finendo quasitilting $R$-modules, see Theorem \ref{Bijection-Finendo}. It is known that all silting modules are finendo quasitilting. 
Dually, we also characterize all support $\tau^{-}$-tilting subcategories of $\Mod R$ by  providing a bijection between the equivalence classes of all support $\tau^-$-tilting subcategories of $\Mod R$ and the collection of all equivalent classes of the quasicotilting $R$-modules. 

Towards the end of the paper, we provide applications of this theory to the category of the representations of a finite and acyclic quiver in an abelian category $\SA$ with enough projective objects. Let $Q$ be such a quiver. It is known that the category $\Rep(Q, \SA)$ of representations of $Q$ over $\SA$ is again an abelian category with enough projective objects. See Section~\ref{Sec:Quivers} for more details.
We give a recipe to construct support $\tau$-tilting subcategories in $\Rep(Q, \SA)$ from certain support $\tau$-tilting subcategories of $\SA$. More explicitly, we have the following result.

\begin{theorem}[Theorem~\ref{ProduceTauTilting}]
Let $\SA$ be an abelian category with enough projective objects and $Q$ be a finite acyclic quiver.
Let $\ST$ be a support $\tau$-tilting subcategory of $\SA$ such that $\Fac(\ST)$ is closed with respect to the kernels of epimorphisms.
Then
\[\mathbb{T} =\add\lbrace e_i^{\rho}(T) \vert ~ i\in Q_0, T\in\ST\rbrace\]
is a support $\tau$-tilting subcategory of $\Rep(Q, \SA)$.
\end{theorem}

Let $\Prj(\SA)$ denote the category of all projective objects in $\SA$. In \cite[Proposition~3.9]{BBOS} the authors have shown that
\[\mathbb{T} =\add\lbrace e_i^{\rho}(P) \vert ~ i\in Q_0, P \in\Prj(\SA) \rbrace\]
is a tilting subcategory of $\SA$.
Then the previous result can be seen as a generalization of this result, since $\SA = \Fac(\Prj(\SA))$ is clearly closed under kernels of epimorphisms.

\medskip
There are other generalisations of tilting theory that we have not mentioned yet.
One of these generalizations consists on defining $n$-tilting subcategories, where $1$-tilting subcategories correspond to classical tilting subcategories, see Definition~\ref{HigherDef(BBOS)}.
By convention we consider $\Prj(\SA)$ as the $0$-tilting subcategory of $\SA$.
Using similar methods as in the proof of Theorem~\ref{ProduceTauTilting} we were able to show another generalization of \cite[Proposition~3.9]{BBOS}, which we believe might have interesting implications in the study of higher homological algebra, introduced by Iyama in \cite{Iyama1, Iyama2}.

\begin{theorem}[Theorem~\ref{ThHigherTilt}]
Let $\SA$ be an abelian category with enough projective objects.
Let $n$ be a non-negative integer and let $Q$ be a finite and acyclic quiver.
For an $n$-tilting subcategory $\ST$ of $\SA$ put
\[\ST' = \add(\lbrace e_i^{\rho}(T) \vert ~ i\in Q_0, T \in \ST\rbrace).\]
Then $\ST'$ is an $(n+1)$-tilting subcategory of $\Rep(Q, \SA)$.
\end{theorem}

Throughout the paper, several examples are provided using the theory of persistent modules, which are central objects of study in topological data analysis (TDA).
We refer the reader to the survey \cite{Ca} by Carlsson, or the book \cite{Ou} by Oudot, for an introduction to topological data analysis and connections to quiver representations.

The paper is structured as follows. In Section~\ref{Sec:Preliminaries} we fix notation and give the necessary background for the rest of the paper.
In Section~\ref{Sec:TauTiltSubcategories} we define support $\tau$-tilting subcategories in abelian categories with enough projective objects and we give some of its basic properties.
Then, in Section~\ref{Sec:TauCotorsionTorsionTriples} we introduce and study the notion of $\tau$-cotorsion torsion triples. We also compare $\tau$-cotorsion torsion triples with the left weak cotorsion torsion triples of \cite{BZ}.
In Section~\ref{Sec:Bijection} we show the bijection between support $\tau$-tilting subcategories and $\tau$-cotorsion torsion triples.
Section~\ref{Sec:Dual} is a compilation of dual definitions and results.
Later, in Section~\ref{Sec:Quadruples} we relate support $\tau$-tilting subcategories with support $\tau^-$-tilting subcategories via $\tau$-$\tau^-$-quadruples.
In Section~\ref{Sec:Connection} we study support $\tau$- and $\tau^-$- tilting subcategories in $\Mod R$, where $R$ is a unitary associative ring, and show that how they are connected to the silting and cosilting theory. We end the paper in Section~\ref{Sec:Quivers} where we study support $\tau$-tilting subcategories and $n$-tilting subcategories in the category $\Rep(Q, \SA)$ of quiver representations over $\SA$.

\section{Preliminaries}\label{Sec:Preliminaries}
Let $\SA$ be an abelian category.
In this paper by subcategory we always mean a full subcategory.
Let $\SX$ be a subcategory of $\SA$.
A morphism $\varphi : X \lrt A$, where $A$ is an object of $\SA$, is called a right $\SX$-approximation of $A$ if, $X \in \SX$ and for every $X' \in \SX$, the induced morphism $\SA(X', X) \lrt \SA(X', A) \lrt 0$ of abelian groups is exact.
We say that $\SX$ is a contravariantly finite subcategory of $\SA$ if every object $A$ of $\SA$ admits a right $\SX$-approximation.
Dually, the notions of left $\SX$-approximations and covariantly finite subcategories are defined.
Moreover we say that $\SX$ is a functorially finite subcategory of $\SA$ if it is both a contravariantly finite and a covariantly finite subcategory of $\SA$.
If $\SX$ is closed under taking finite direct sums and direct summands we say that it is additively closed.

Let $n$ be a non-negative integer. We define
\[ \SX^{\perp_n}:=\lbrace A\in\SA ~ \vert ~\Ext^n_\SA(\SX, A)=0\rbrace,\]
\[ {}^{\perp_n}\SX:=\lbrace A\in\SA ~ \vert ~\Ext^n_\SA( A, \SX)=0 \rbrace.\]
Note that $\Ext^0$ is just the usual Hom-functor.

Moreover, we set
\[ \SX^{\perp}:=\lbrace A\in\SA ~ \vert ~\Ext^i_\SA(\SX, A)=0, ~~~\forall~ i\geq 1\rbrace,\]
\[ {}^{\perp}\SX:=\lbrace A\in\SA ~ \vert ~\Ext^i_\SA( A, \SX)=0, ~~~\forall~ i\geq 1\rbrace.\]

\begin{definition}\label{HigherDef(BBOS)}
Let $\SA$ be an abelian category with enough projective objects. Let $n$ be a non-negative integer.
An additively closed subcategory $\ST$ of $\SA$ is called an $n$-tilting subcategory if
\begin{itemize}
\item[$(i)$] $\ST$ is a contravariantly finite subcategory of $\SA$.
\item[$(ii)$] $\Ext^i_\SA(T_1, T_2)=0$, for all $T_1, T_2\in\ST$ and all $i\geq 1$.
\item[$(iii)$] Every object $T\in\ST$ has projective dimension at most $n$.
\item[$(iv)$] For every projective object $P$ in $\SA$, there exists a short exact sequence
\[0\lrt P\st{f}{\lrt}T^0\lrt T^1\lrt  \cdots\lrt T^n\lrt 0\]
with $T^i\in\ST$.
\end{itemize}
If $\ST$ only satisfies the conditions $(ii)-(iv)$, it is called a weak $n$-tilting subcategory of $\SA$.
\end{definition}

\begin{remark}
It follows, by breaking out the exact sequence of $(iv)$ to short exact sequences and using the vanishing of Ext-groups in $(ii)$ of the previous definition, that the map $f: P \lrt T^0$ is a left $\ST$-approximation of $P$. Moreover, it is easy to see  that the category $\Prj(\SA)$ of all projective objects in $\SA$ is the only $0$-tilting subcategory of $\SA$. In this paper, when we say tilting subcategory we mean a $1$-tilting subcategory.
\end{remark}

\s {\sc Cotorsion torsion triple.} Here we recall the notion of torsion pairs, cotorsion pairs and cotorsion torsion triples. Let us begin by the following classical definition of Dickson \cite{Dickson}.

\begin{definition}
Let $\SA$ be an abelian category.
A pair $(\ST,\SF)$ of full subcategories of $\SA$ is called a torsion pair if $\Hom_{\SA}(\ST,\SF)=0$ and for every $A \in \SA$ there is a short exact sequence
$$0 \lrt tA \lrt A \lrt fA \lrt 0$$
such that $tA \in \ST$ and $fA \in \SF$.
\end{definition}

It is well known that $\ST$ and $\SF$ determine each others and the short exact sequence is functorial, see for instance \cite[\S I.1]{BR}.

\begin{definition}(see \cite{Sa})
Let $\SA$ be an abelian category with enough projective objects.
A pair $(\SC, \SD)$ of full subcategories of $\SA$ is called a cotorsion pair if $\SC = {}^{\perp_1}\SD$, $\SD = \SC^{\perp_1}$ and for every object $A \in \SA$, there are short exact sequences
\[0 \lrt D \lrt C \lrt A \lrt 0 \]
\[ 0 \lrt A \lrt D' \lrt C' \lrt 0\]
where $C$ and $C'$ are in $\SC$ and $D$ and $D'$ are in $\SD$.
\end{definition}
This notion sometimes called complete cotorsion pair in the literature.
See \cite[Remark~2.7]{BBOS} for more details.

\begin{remark}
It is immediate from the definition that in a cotorsion pair $(\SC, \SD)$ both $\SC$ and $\SD$ are closed under extensions.
\end{remark}

Let $\SA$ be an abelian category with enough projective and enough injective objects.
In view of \cite[ Lemma V.3.3]{BR}, a pair $(\SC, \SD)$ of full subcategories of $\SA$ is a cotorsion pair if and only if the following conditions are satisfied:
\begin{itemize}
\item[$1.$] $\SC = {}^{\perp_1}\SD$,
\item[$2.$] for every object $A \in \SA$, there exists a short exact sequence
\[ 0 \lrt A \st{f}\lrt D \lrt C \lrt 0,\]
 where $D \in \SD$ and $C \in \SC$.
\end{itemize}

Following proposition shows that when $\SD$ is closed under factors it is enough, in the above definition, to check the Condition $2$ only for projective objects.

\begin{proposition}\label{motivation}
Let $\SA$ be an abelian category with enough projective and enough injective objects.
Let $(\SC, \SD)$ be a pair of full subcategories of $\SA$ such that $\SD$ is closed under factors.
Then $(\SC, \SD)$ is a cotorsion pair if and only if
\begin{itemize}
\item[$1.$] $\SC = {}^{\perp_1}\SD$,
\item[$2.$] for every projective object $P \in \SA$, there exists a short exact sequence
\[ 0 \lrt P \st{f}\lrt D \lrt C \lrt 0,\]
where $D \in \SD\cap\SC$ and $C \in \SC$.
\end{itemize}
\end{proposition}

\begin{proof}
We start by showing the necessity.
Let $A$ be an arbitrary object of $\SA$. Since $\SA$ has enough projective objects, we have an epimorphism $P \lrt A \lrt 0$.
By assumption $P$ fits into the short exact sequence
\[ 0 \lrt P \st{f}\lrt D \lrt C \lrt 0.\]
The pushout diagram
\[\xymatrix{
 0 \ar[r] & P\ar[d]^{f} \ar[r] & D \ar[d]^{g} \ar[r] &C \ar@{=}[d]\ar[r] & 0 \\
 0 \ar[r] & A \ar[d] \ar[r] & {U} \ar[d]\ar[r] &C \ar[r] & 0,\\
 &  0 & 0  }
\]
induces the short exact sequence $0 \lrt A \lrt U \lrt C \lrt 0$, where $U \in \Fac(\SD)$.
Since $\SD$ is closed under factors, we deduce that $U \in \SD$ and hence this is the desired short exact sequence for $A$.

For the sufficiency, let $P$ be a projective object in $\SA$.
Then there exists a short exact sequence
\[ 0 \lrt P \st{f}\lrt D \lrt C \lrt 0\]
where $D \in \SD$ and $C \in \SC$.
Since $P$ is projective, then $P \in {}^{\perp_1}\SD = \SC$.
Hence $D \in \SC$ because $\SC$ is closed under extensions and the result follows.
\end{proof}

Following lemma shows that if the subcategory $\SC$ of a cotorsion pair $(\SC, \SD)$ of $\SA$ is closed with respect to factors, then $\SC \cap \SD$ is contravariantly finite in $\SA$.

\begin{lemma}\label{ContraFinite}
{Let $(\SC, \SD)$ be a cotorsion pair in $\SA$ such that $\SD$ is closed under factors. Then $\SC \cap \SD$ is a contravariantly finite subcategory of $\SA$. }
\end{lemma}

\begin{proof}
Since $(\SC, \SD)$ is a cotorsion pair, it follows easily that $\SD$ is closed under extensions and products. So $\SD$ is indeed a torsion class.
Now, since $(\SC, \SD)$ is a cotorsion pair, there exists a short exact sequence
\[ 0\lrt D \lrt C\st{\imath}\lrt tA\lrt 0\]
such that $D \in\SD$ and $C\in\SC$. Here, since $D, tA\in\ST$ and $\ST$ is closed under extensions, we get $C \in \SC \cap \SD$. We claim that the composition $C \st{\imath}\lrt tA\st{\jmath}\lrt A $ is a right $\SC \cap \SD$-approximation. To show the claim, let $X\st{\ell}\lrt A$ be a morphism with $X \in \SC \cap \SD$. Since $X \in \SD$ and every map from $\SD$ to $A$, factors through $tA$, the morphism $\ell$ factors through $\jmath$. Now since $X \in \SC$ and $\Ext^1_A(X, D)=0$, the morphism $\ell$ factors through $\jmath\imath$. Hence the claim is proved.
 \end{proof}

\begin{definition}
Let $\SA$ be an abelian category with enough projective objects.
A triple $(\SC, \ST, \SF)$  of full subcategories in $\SA$ is called a {cotorsion torsion triple}, if $(\SC, \ST)$ is a cotorsion pair and $(\ST, \SF)$ is a torsion pair.
\end{definition}

It has been recently shown in \cite{BBOS} that cotorsion torsion triples and tilting subcategories are closely related.
The explicit relation between these two concepts is as follows.

\begin{theorem}\cite[Theorem 2.29]{BBOS}\label{2.29}
Let $\SA$ be an abelian category with enough projective objects. Then there is a bijection

\[\lbrace \mbox{tilting subcategories}\rbrace ~~~~\longleftrightarrow~~~~ \lbrace \mbox{cotorsion torsion triples}\rbrace\]
 $\  \  \  \  \  \ \ \ \ \ \ \ \ \ \ \ \ \ \ \ \ \ \ \ \ \  \ \ \ \ \ \ \ \ \ \ \ \ \ \ \ \ \  \ \ \ \ \ \ \ST~~~~\longrightarrow~~~~~({}^{\perp_1}{(\rm{Fac}(\ST))}, \rm{Fac}(\ST), \ST^{\perp_0})$
\[\SC\cap\ST~~\longleftarrow ~~~(\SC, \ST, \SF)\]
where $\rm{Fac}(\ST)$ is the  full subcategory of $\SA$ consisting of factors of objects in $\ST$.
\end{theorem}

\s {\sc Pointwise finite dimensional representations.}
Let $k$ be a field and $(\CX, \leq)$ be a poset.
Consider $\CX$ as a category.
A persistence module is a (covariant and additive) functor $V: \CX \to \Mod k$ from $\CX$ to $\Mod k$, the category of $k$-vector spaces.
A persistence module $V$ is called pointwise finite dimensional representation of $\CX$ if it is a functor from $\CX$ to $\mmod k$, the category of finite dimensional vector spaces.

The abelian category of pointwise finite dimensional representations will be denoted by $\Rep_k^{\rm{pfd}}\CX$.
A subset $\CC$ of $\CX$ is called a convex subset if for every $x \leq y\in\CC$, $x\leq z\leq y$ implies $z\in\CC$. Let $\CC$ be a convex subset of $\CX$. The constant representation $k_\CC$ is defined by $k_\CC(x)=k$, for $x\in\CC$, $k_\CC(x)=0$, for $x\not\in \CC$ and $k_\CC(x\leq y)=\rm{id}_k$.
Set $P_x=k_{\lbrace y\in\CX \vert \ y\geq x\rbrace}.$
As application of Yoneda's Lemma one can see that, for every $x \in \CX$, $P_x$ is a projective object of $\Rep_k^{\rm{pfd}}\CX$ and vice versa all projective modules are of this form, see e.g. \cite[\S 2]{BBOS}.
A representation $V\in\Rep_k^{\rm{pfd}} \CX$ is called finitely generated if there exists an epimorphism of functors $\bigoplus_{i\in J} P_{x_i}\lrt V$, where $J$ is a finite indexing set.
Moreover, if the kernel of this epimorphism is again finitely generated, it is called finitely presented.
We denote the full subcategory of $\Rep_k^{\rm{pfd}} \CX$ consisting of finitely presented representations by $\Rep_k^{\rm{fd}}\CX$.

Of particular importance is the case $\CX=\mathbb{R}_{\geq 0}$, the poset of non-negative real numbers and to concentrate on the abelian subcategory $\RepR$ of finitely presented representations.
It is proved in \cite[Theorem 1.1]{CB} that the indecomposable objects in this category are classified by the constant representation $k_{ [x, y)}$, where $x< y\leq \infty$.
Moreover, for each $x\geq 0$, the constant representations $k_{[ x, \infty)}$ and $k_{[ 0, x)}$ are projective and injective objects, respectively.

\section{$\tau$-tilting subcategories}\label{Sec:TauTiltSubcategories}
In this section, we recall the definition of $\tau$-tilting subcategories and study some of their properties. Following definition is motivated by \cite[Definition 1.5]{IJY} and \cite[Definition 2.1]{LZh}.

\begin{definition}\label{DefTauTilting}
Let $\SA$ be an abelian category with enough projective objects.
Let $\ST$ be an additive full subcategory of $\SA$.
Then $\ST$ is called a  weak support $\tau$-tilting subcategory of $\SA$ if
\begin{itemize}
  \item[$(i)$] $\Ext^1_{\SA}(T_1,\Fac(T_2))=0,$ for all $T_1, T_2 \in \ST$.
  \item[$(ii)$] For any projective $P$ in $\SA$, there exists an exact sequence
  \[P \st{f}{\lrt} T^0 \lrt T^1 \lrt 0\]
  such that $T^0$ and $T^1$ are in $\ST$ and $f$ is a left $\ST$-approximation of $P$.
\end{itemize}
If furthermore $\ST$ is a contravariantly finite subcategory of $\SA$, it is called a support $\tau$-tilting subcategory of $\SA$.
A support $\tau$-tilting subcategory is simply called $\tau$-tilting if the approximation $f: P \lrt T^0$ is non-zero for every projective object $P$.
\end{definition}

\begin{remark}
Note that a key difference between tilting and $\tau$-tilting subcategories is that the approximations $f: P \lrt T^0$ are always a monomorphism when $\ST$ is a tilting subcategory.
It is immediate from the definition that every (weak)  tilting subcategory is a  (weak) $\tau$-tilting subcategory.
\end{remark}

In the following example we present a support $\tau$-tilting subcategory which is not a tilting subcategory. Note that for $X\in\SA$, we let $\add(X)$ be the category of all direct summands of finite direct sums of copies of $X$.

\begin{example}\label{TauTiltingCatExam1}
Let $\SA=\RepR$, and set
\[\ST=\add( \lbrace k_{[x, \infty)} \vert \ x\geq 1\rbrace \cup \lbrace k_{[0, x)}\vert  \ x\leq 1\rbrace).\]
Then $\ST$ is a support $\tau$-tilting subcategory. To see this, first note that,
\[\Fac(\ST)=\add( \lbrace k_{[x, y)} \vert \ 1\leq x< y\leq\infty\rbrace \cup \lbrace k_{[0, x)}\vert  \ x\leq 1\rbrace),\]
and  obviously $\Ext^1_\SA(\ST, \Fac(\ST))=0$.

Now let $k_{[a, \infty)}$ be an indecomposable projective in $\SA$.
If $a\geq 1$, the exact sequence
\[0\lrt k_{[a, \infty)}\lrt k_{[a, \infty)}\lrt 0,\]
and if $0\leq a< 1$, the exact sequence
\[k_{[a, \infty)}\lrt k_{[0, 1)} \lrt k_{[1, \infty)}\lrt 0,\]
 are the desired exact sequences.

Moreover, $\ST$ is a contravariantly finite subcategory of $\SA$. Indeed, depending on $a$ and $b$, for an indecomposable representation $k_{[ a, b)}$ we have the following right $\ST$-approximations
 \[
 \left\{
 \begin{array}{ll}
 \  k_{[0, b)}\lrt k_{[0, b)},  & \  0= a<  b \leq 1;\\
k_{[1, \infty)}\lrt k_{[a, b)},  & \  0\leq  a< 1 < b \leq \infty;\\
\ \ \ \ \ \ 0\lrt k_{[a, b)},  & \  0<  a<  b \leq 1;\\
k_{[a, \infty)}\lrt k_{[a, b)}, & \   1\leq a < b\leq \infty.
 \end{array}
 \right. \]
It is obviously not a tilting subcategory of $\SA$, because $k_{[0, \infty)}$ is a projective object, for which we do not have a short exact sequence like Condition $(iv)$ of Definition \cite{BBOS}.
\end{example}

One of the main reasons behind the success of $\tau$-tilting theory \cite{AIR} is the formal inclusion of the notion of \textit{support} from the start, see Definition~\ref{DefAIR}.
Note that in this definition we are defining objects in $\mmod \Lambda$ by using a property of the object in a different category, namely $\mmod (\Lambda/\langle e\rangle)$.
So one needs to verify that the good properties of a support $\tau$-tilting object $M$ in $\mmod (\Lambda/\langle e\rangle)$ can be transported to $\mmod \Lambda$.
This was done in \cite[Lemma~2.1]{AIR}.
Note that $\mmod (\Lambda/\langle e\rangle)$ is a functorially finite wide subcategory of $\mmod \La$ which is at the same time a torsion and a torsion-free class. Recall that a subcategory $\SX$ of an abelian category $\SA$ is called a wide subcategory if it is closed under kernels, cokernels and extensions. In particular, this implies that $\SX$ itself is an abelian category.
It is known that a subcategory $\SX$ of an abelian category $\SA$ is a torsion class if it is a contravariantly finite subcategory of $\SA$, closed with respect to quotients and extensions. Moreover, $\SX$ is a torsion free class if it is a covariantly finite subcategory of $\SA$ which is furthermore closed under subobjects and extensions.

\begin{lemma}\label{TorsionTorsionfree}
Let $\SA$ be an abelian category with enough projective objects. Let $\SX$ be a wide and functorially finite  full subcategory of $\SA$. Then $\SX$ is a torsion class  of $\SA$ if and only if $\SX$ is a torsion free class of $\SA$.
\end{lemma}

\begin{proof}
Let $\SX$ be a torsion class of $\SA$.  In order to show that it is a torsion free class, we need to show that it is covariantly finite,  closed under subobject and  closed  under extension. By the assumption $\SX$ is  a covariantly finite subcategory. It is closed under extension, since $\SX$ is a torsion class. Moreover, it is closed under subobject, since it is a wide subcategory.
The other implication follows similarly.
\end{proof}

Our next result is inspired by \cite[Lemma~2.1]{AIR} and it justifies the name of support $\tau$-tilting subcategories.

\begin{proposition}\label{ProducingTauTiltingFromTiltingRep}
Let $\SA$ be an abelian category with enough projective objects. Let $\SX$ be a wide and functorially finite  torsion class of $\SA$. Then every support $\tau$-tilting subcategory of $\SX$ is a support $\tau$-tilting subcategory of $\SA$.
\end{proposition}

\begin{proof}
Let $\ST$ be a support $\tau$-tilting subcategory of $\SX$. First we note that, since $\SX$ is a wide subcategory, then $\Fac(\ST)\subseteq \Fac(\SX)=\SX$. Therefore, since $\ST$ is a support $\tau$-tilting subcategory of $\SX$, we have $\Ext^1_\SX(\ST, \Fac(\ST))=0$. Hence $\Ext^1_\SA(\ST, \Fac(\ST))=0$.

Now let $P$ be a projective object in $\SA$. By Lemma \ref{TorsionTorsionfree}, there is a subcategory $\SY$ of $\SA$ such that $(\SY, \SX)$ is a torsion pair of $\SA$. Consider the canonical short exact sequence
\[0\lrt Y\lrt P\st{f}\lrt X\lrt 0\]
of $P$ with respect to this torsion pair. So $Y \in \SY$ and $X \in \SX$. Let $X'$ be an arbitrary object of $\SX$. By applying the functor $\Hom_\SA(-, X')$ to the above short exact sequence, we get $\Ext^1_\SA(X, X')=0$ and therefore $\Ext^1_\SX(X, X')=0$. Hence $X$ is an Ext-projective object in $\SX$. Since $\ST$ is a  support $\tau$-tilting subcategory of $\SX$, there exists  an  exact sequence
\[ X\st{g}\lrt T^0\lrt T^1\lrt 0\]
where $T^0, T^1\in\ST$ and $g$ is a left $\ST$-approximation. Now  the exact sequence
\[P\st{gf}\lrt T^0\lrt T^1\lrt 0\]
is the desired one. In fact, it is easy to see  that  $gf$ is a left $\ST$-approximation of $P$.

Finally we show that $\ST$ is a contravariantly finite subcategory of $\SA$. Let $A\in \SA$. Since $\SX$ is a functorially finite subcategory of $\SA$, there exists a right $\SX$-approximation $X\st{f}\lrt A$ of $A$. Consider a right $\ST$-approximation $T\st{g}\lrt X$ of $X$, which exists because $\ST$ is a support $\tau$-tilting subcategory of $\SA$. Now it is easy to see that $T\st{fg}\lrt A$ is a right $\ST$-approximation of $A$.
\end{proof}

We illustrate our previous result with the following example.
\begin{example}\label{TauTiltingCatExam2}
Let $\SA=\RepR$ and $\SB=\RepRO$.
On one hand, since $\mathbb{R}_{\geq 0}$ and $\mathbb{R}_{\geq 1}$ are isomorphic as posets, it is clear that $\SA$ and $\SB$ are equivalent as categories.
On the other hand, the natural inclusion of $\mathbb{R}_{\geq 1}$ into $\mathbb{R}_{\geq 0}$ induces an embedding of $\SB$ into $\SA$.
In fact, it is easy to see that $\SB$ is a functorially finite wide subcategory of $\SA$ which is both a torsion and a torsion free class.
Set
\[\ST=\add(\lbrace k_{[ 1, y)} \vert \  1<y\leq \infty\rbrace).\]
It is easy to see that $\ST$ is a tilting subcategory of $\SB$  and hence a support $\tau$-tilting subcategory of $\SB$.
So, Proposition~\ref{ProducingTauTiltingFromTiltingRep} implies that $\ST$ is a support $\tau$-tilting subcategory of $\SA$.
Note that $\ST$ is not a tilting subcategory of $\SA$ since the $\ST$-approximation $f: k_{[0,\infty)} \to  k_{[1,\infty)}$ is not a monomorphism, in fact it is  a zero morphism.
\end{example}

By \cite[Proposition 3.42]{R}, if $\SA$ is an abelian category with enough projective objects  such that $\Prj(\SA)=\add(P)$, then for every tilting subcategory $\ST$ of $\SA$ there exists an object $T \in \ST$ such that $\ST = \add(T)$. In fact $T$ is a tilting object. Note that an object $T\in\SA$ is called a tilting object if
\begin{itemize}
\item[$(i)$] $\Ext^1_\SA(T, T)=0$.
\item[$(ii)$] Projective dimension of $T$ is at most one.
\item[$(iii)$ ] For every projective object $P\in \SA$, there exists a short exact sequence
\[0\lrt P\lrt T^0\lrt T^1\lrt 0\]
such that $T^0, T^1\in\add(T)$.
\end{itemize}

We prove a version of this fact for $\tau$-tilting subcategories. To do this we provide the following definition of a $\tau$-tilting object in an abelian category.

\begin{definition}\label{TauTiltingObjectDef}
Let $\SA$ be an abelian category with enough projective objects. An object $T\in \SA$ is called a  support $\tau$-tilting object if
\begin{itemize}
\item[$(i)$] $\Ext^1_\SA(T, \Fac(T))=0$.
\item[$(ii)$] For every projective object $P\in\SA$, there exists an exact sequence
\[P\st{f}\lrt T^0\lrt T^1\lrt 0\]
such that $T^0, T^1\in \add(T)$ and $f$ is a left $\add(T)$-approximation.
\end{itemize}
\end{definition}

Let $\SA$ be an abelian category with enough projective objects. Let $T$ be a support $\tau$-tilting object in $\SA$.
It follows directly from the definition that $\add(T)$ is a weak support $\tau$-tilting subcategory of $\SA$.
Next proposition provides a partial converse to this fact.

\begin{remark}\label{Jasso2.14}
We note that the previous definition is not the classical definition of support $\tau$-tilting modules in $\mmod\La$ introduced in \cite{AIR}.
However it follows from \cite[Proposition~2.14]{J} that these two definitions coincide if $\SA=\mmod\La$.
\end{remark}

\begin{proposition}\label{addgentautilt}
Let $\SA$ be an abelian category with enough projective objects such that $\Prj(\SA)=\add(P)$, for some object $P\in\SA$.
If $\ST$ is a support $\tau$-tilting subcategory of $\SA$, then there exists a support $\tau$-tilting object $T\in \ST$ such that $\Fac(\ST)=\Fac(T)$.
\end{proposition}

\begin{proof}
Consider the exact sequence \[P\lrt T^0\lrt T^1\lrt 0\] which  exists, because $\ST$ is a support $\tau$-tilting subcategory.
We claim that $\Fac(T^0\oplus T^1)=\Fac(\ST)$.
First we prove that $T^0\oplus T^1$ is a  support $\tau$-tilting object.
To this end, we just need to show that  every projective object $P'\in \Prj (\SA)=\add(P)$ admits an exact sequence \[P'\st{f'}\lrt T^0_{P'}\lrt T^1_{P'}\lrt 0\]
where $f'$ is a  left $\add(T^0\oplus T^1)$-approximation and $T^0_{P'} , T^1_{P'}\in\add(T^0\oplus T^1)$.
Since $\Prj(\SA)=\add(P)$, there exists a non-negative integer $n$ and projective object $Q\in\add(P)$ such that $P'\oplus Q=P^n$.
So we get the exact sequence \[P'\oplus Q \st{(g \ h)}\lrt (T^0)^n\lrt (T^1)^n\lrt 0.\]
Based on this sequence, we can construct the following commutative diagram
\[\xymatrix{
  & 0 \ar[d]&0\ar[d] &\\
  & (T^0)^n\ar@{=}[r]\ar[d]& (T^0)^n\ar[d]\\
 {P'\oplus Q}\ar@{=}[d]\ar[r]^-{g\oplus h} & (T^0)^n\oplus (T^0)^n \ar[d] \ar[r] & \Coker g\oplus\Coker h \ar[d]\ar[r] & 0 \\
  {P'\oplus Q}\ar[r]^{(g \ h)} & (T^0)^n  \ar[r]\ar[d]&(T^1)^n \ar[r]\ar[d] & 0 \\
 & 0&0
}\]
Since the second vertical short exact sequence splits, both $\Coker g $ and $\Coker h $ are in $\add(T_0\oplus T_1)$. Therefore the  exact sequence
\[ P'\lrt (T^0)^n\lrt \Coker g \lrt 0\]
is the desired one.

Now we show that $\Fac(\ST) = \Fac(T^0\oplus T^1)$.
It is clear that $\Fac(T^0\oplus T^1)\subseteq \Fac(\ST)$.
So it is enough to show the reverse inclusion.
To do this, we show that every object $\bar{T}\in\ST$ lies in $\Fac(T^0\oplus T^1)$.
Let \[P_1\st{h_1}\lrt P_0\st{h_0}\lrt \bar{T}\lrt 0\] be a projective presentation of $\bar{T}$.
Since $T^0\oplus T^1$ is a  support $\tau$-tilting object, for $i=0, 1$, there are exact sequences
\[P_i\st{f_i}\lrt T^0_i\st{g_i}\lrt T^1_i\lrt 0\]
with $T^0_i , T^1_i\in\add(T_0\oplus T_1)$.
Therefore, since $f_i$ is a left $\add(T^0\oplus T^1)$-approximation of $P_i$, we have the commutative diagram
\[\xymatrix{
  & P_1\ar[r]^{h_1}\ar[d]^{f_1}& P_0\ar[d]^{f_0}\ar[r]^{h_0}& \bar{T}\ar@{=}[d]\ar[r]&0\\
& T^0_1\ar[r]^{k_1}\ar[d]^{g_1}& T^0_0\ar[r]^{k_0}\ar[d]^{g_0}& \bar{T}\\
&T^1_1\ar[d]& T^1_0\ar[d]&\\
&0&0
}\]
where the first row and the first column are exact.
Since $g_0k_1f_1=0$ and $k_0k_1f_1=0$, by the cokernel property we have morphisms  $l_1: T^1_1\lrt T^1_0$ and $l_0: T^1_1\lrt \bar{T}$ such that $g_0k_1=l_1g_1$ and $l_0g_1=k_0k_1$.
To finish the proof we show that the sequence
\[\begin{tikzcd}
 T^0_0\oplus T^1_1\ar{r}{\begin{bmatrix}  g_0\  l_0\\ k_0 \  l_1 \end{bmatrix}}& T^1_0\oplus \bar{T}\ar{r} &0
\end{tikzcd}\]
is exact.
Indeed, let $x\oplus y\in T^1_0\oplus \bar{T}$.
Since $g_0$ is an epimorphism, there exists $t'_0\in T^0_0$ such that $g_0(t'_0)=x$.
Now since $h_0$ is an epimorphism, there exists $p_0\in P_0$ such that $h_0(p_0)=k_0(t'_0)-y$.  Consider  $ (t'_0-f_0(p_0))\oplus 0\in T^0_0\oplus T^1_1$.
Therefore, it is easy to see that
\[{\begin{bmatrix}  g_0\  l_0\\ k_0 \  l_1 \end{bmatrix}} \begin{bmatrix} t'_0-f_0(p_0) \\ 0 \end{bmatrix} =\begin{bmatrix}
x\\y
\end{bmatrix}.\]
Hence $\bar{T}\in \Fac(T^0_0\oplus T^1_1)\subseteq \Fac(T_0\oplus T_1)$.
\end{proof}

When $\ST$ is a weak tilting subcategory of $\SA$, then $\Fac(\ST) = \ST^{\perp_1}$, see \cite[Proposition 2.22]{BBOS}. The following example shows that this equality does not hold in general for $\tau$-tilting subcategories.

\begin{example}
Let $A$  be the path algebra of the quiver
$$
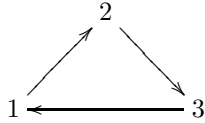
\xymatrix{
  & 2\ar[dr]& \\
  1\ar[ru] & & 3\ar[ll] }$$
modulo the ideal generated by all paths of length 2. The Auslander-Reiten quiver of $A$ is
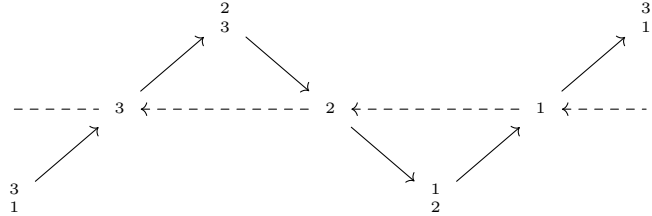
\begin{figure}[h]
    \centering
			\begin{tikzpicture}[line cap=round,line join=round ,x=1.4cm,y=1.2cm]
					\clip(-1.4,-0.3) rectangle (5.3,2.5);
					\draw [->] (-0.8,0.2) -- (-0.2,0.8);
					\draw [->] (3.2,0.2) -- (3.8,0.8);
					\draw [->] (0.2,1.2) -- (0.8,1.8);
					\draw [->] (4.2,1.2) -- (4.8,1.8);
					\draw [<-, dashed] (4.2,1.0) -- (5,1.0);
					\draw [dashed] (-1.,1.0) -- (-0.2,1.0);
					\draw [<-, dashed] (0.2,1.0) -- (1.8,1.0);
					\draw [<-, dashed] (2.2,1.0) -- (3.8,1.0);
					\draw [->] (2.2,0.8) -- (2.8,0.2);
					\draw [->] (1.2,1.8) -- (1.8,1.2);

				\begin{scriptsize}
					\draw (-1,0) node {$\Dim{3\\1}$};
					\draw (5,2) node {$\Dim{3\\1}$};
					\draw (3,0) node {$\Dim{1\\2}$};
					\draw (4,1) node {$\Dim{1}$};
					\draw (0,1) node {$\Dim{3}$};
					\draw (2,1) node {$\Dim{2}$};
					\draw (1,2) node {$\Dim{2\\3}$};

				\end{scriptsize}
			\end{tikzpicture}
	\end{figure}
	
Let $\ST=\add(\Dim{2\\3} \oplus \Dim{2} \oplus \Dim{1\\2})$. Then $\ST$ is a $\tau$-tilting subcategory of $\mmod A$.  It is easy to see that
\[\Fac(\ST)= \add(\Dim{2\\3} \oplus \Dim{2} \oplus \Dim{1\\2} \oplus \Dim{1})\neq \add(\Dim{2\\3} \oplus \Dim{2} \oplus \Dim{1\\2} \oplus\Dim{1}\oplus \Dim{3\\1})=\ST^{\perp_1}.\]
\end{example}

\section{$\tau$-cotorsion torsion triples}\label{Sec:TauCotorsionTorsionTriples}
In this section we introduce a triple of full additive subcategories of $\SA$, that will be called a $\tau$-cotorsion torsion triple, or simply a $\tau$-triple, and show that there is a bijection between the collection of all $\tau$-tilting subcategories of $\SA$ and the collection of all $\tau$-triples.
The Proposition~\ref{motivation} motivates the following definition.

\begin{definition}\label{WeakCotorsion}
Let $\SA$ be an abelian category with enough projective objects.
A pair of full subcategories $(\SC, \SD)$ of $\SA$ is called a $\tau$-cotorsion pair if
\begin{itemize}
\item[$1.$] $\SC = {}^{\perp_1}\SD$.
\item[$2.$] For every projective object $P \in \SA$, there exists an exact sequence
\[  P \st{f}\lrt {D}\lrt C \lrt 0,\]
 where $D \in \SC\cap\SD$, $C \in \SC$ and $f$ is a left $\SD$-approximation.
\item[$3.$] $\SC\cap\SD$ is a contravariantly finite subcategory of $\SA$.
\end{itemize}
\end{definition}

\begin{example}\label{TauCotorsionExample}
Let $\SA=\RepR$ and set
\begin{align*}
\SC=&\add(\lbrace k_{[x, \infty)} \vert \ 0\leq x < \infty\rbrace \cup \lbrace k_{[x, y)} \vert \ 0 \leq x < y< 1\rbrace)\\
\SD=&\add(\lbrace k_{[x, y)} \vert \ 1\leq x< y \leq \infty \rbrace).
\end{align*}
Then $(\SC, \SD)$ is a $\tau$-cotorsion pair.
Note that it is not a cotorsion pair, because $\SD$ does not contain the injective representations.
To see it is a $\tau$-cotorsion pair, first note that
 \[\SC\cap \SD=\add(\lbrace k_{[x, \infty)} \vert x\geq 1\rbrace)\]  is a contravariantly finite subcategory of $\SA$.
 Indeed,  every  indecomposable representation $k_{[ a, b)}$ admits a right $\SC\cap\SD$-approximation as follows
 \[
 \left\{
 \begin{array}{ll}
 k_{[a, \infty)}\lrt k_{[a, b)}, & \   1\leq a < b\leq \infty;\\
 k_{[1, \infty)}\lrt k_{[a, b)},  & \  0< a< 1 < b \leq \infty;\\
\ \ \ \ \ \ \ 0\lrt k_{[a, b)},  & \  0\leq a<  b \leq 1.
 \end{array}
 \right.
 \]
Now let $k_{[a, \infty)}$ be an indecomposable projective in $\SA$.
If $a\geq 1$, the exact sequence
 \[0\lrt k_{[a, \infty)}\lrt k_{[a, \infty)}\lrt 0,\]
 and if $a< 1$, the exact sequence
 \[ k_{[a, \infty)}\st{0}\lrt k_{[b, \infty)} \lrt k_{[b, \infty)}\lrt 0\]
with $b\geq 1$, are the desired exace sequences.
Finally, it is straightforward to see that $\SC={}^{\perp_1}\SD$.
\end{example}

It is shown in \cite[Lemma 2.12]{BBOS} that if $\SC$ and $\SD$ are two subcategories of an abelian category $\SA$ such that
\begin{itemize}
\item[$(1)$] $\Ext^1_\SA(\SC, \SD)=0$;
\item[$(2)$] $\SD$ is closed under factor objects;
\item[$(3)$] For every object $A\in\SA$, there exists a short exact sequence  $0 \lrt A \st{\varphi}\lrt D \lrt C \lrt 0$ where $D \in \SD$ and $C \in \SC$,
\end{itemize}
then every object in $\SC$ is of projective dimension at most one. In particular, a cotorsion pair $(\SC, \ST)$ satisfies all the above conditions when it is embedded in a cotorsion torsion triple $(\SC, \ST, \SF)$. So every object in $\SC$ has projective dimension at most one.

The following example shows that the injectivity of the morphism $\varphi$ in Condition $(3)$ is essential.
We note that this is implicit in the proof of \cite[Lemma 2.12]{BBOS}.
Also note that it follows automatically from the previous definition that the map $\varphi$ is a left $\SD$-approximation.

\begin{example}
Let $A$ be the path algebra of the quiver
\[\xymatrix{1\ar[r] &  2\ar[r]& 3}\]
modulo the ideal generated by the composition of the two arrows. The Auslander-Reiten quiver of $A$ is the following.

{\footnotesize
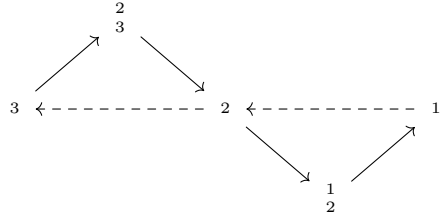
{
\begin{center}
			\begin{tikzpicture}[line cap=round,line join=round ,x=1.4cm,y=1.2cm]
				\clip(-0.4,-0.3) rectangle (4.3,2.5);

					\draw [->] (3.2,0.2) -- (3.8,0.8);
					\draw [->] (0.2,1.2) -- (0.8,1.8);
					\draw [<-, dashed] (0.2,1.0) -- (1.8,1.0);
					\draw [<-, dashed] (2.2,1.0) -- (3.8,1.0);
					\draw [->] (2.2,0.8) -- (2.8,0.2);
					\draw [->] (1.2,1.8) -- (1.8,1.2);
	
				\begin{scriptsize}
					\draw (3,0) node {$\Dim{1\\2}$};
					\draw (4,1) node {$\Dim{1}$};
					\draw (0,1) node {$\Dim{3}$};
					\draw (2,1) node {$\Dim{2}$};
					\draw (1,2) node {$\Dim{2\\3}$};
				\end{scriptsize}
			\end{tikzpicture}
\end{center}}}

Let $\ST=\add(\Dim{1\\ 2}\oplus \Dim{1}\oplus \Dim{3})$.
It is easy to see that $\ST$ is a $\tau$-tilting subcategory of $\mmod A$.
Then we have
\begin{align*}
 \SD & =  \Fac(\ST)=\add(\Dim{1\\ 2}\oplus \Dim{1}\oplus \Dim{3}), \\
 \SC & ={}^{\perp_1}\Fac(\ST)=\add(\Dim{1\\ 2}\oplus \Dim{1}\oplus \Dim{3}\oplus \Dim{2 \\ 3}).
\end{align*}

It is clear that $\Fac(\ST)$ is closed under factor modules.
However the object $\Dim{1} \in \SC$ has projective dimension $2$.
\end{example}

\begin{definition}
A triple $(\SC, \ST, \SF)$ of full subcategories in $\SA$ is called a $\tau$-cotorsion torsion triple, or simply a $\tau$-triple, if $(\SC, \ST)$ is a $\tau$-cotorsion pair and $(\ST, \SF)$ is a torsion pair.
\end{definition}

\begin{example}\label{TauTripleExample}
Let $\SA=\RepR$ and $(\SC, \SD)$ be the $\tau$-cotorsion pair as in the Example \ref{TauCotorsionExample}. Since $\SD$ is a contravariantly finite subcategory of $\SA$ which is  closed under factors and extensions, it is a torsion class. Hence $(\SC, \SD, \SD^{\perp_0})$ is a $\tau$-triple.
\end{example}

\begin{proposition}\label{CoTorImpliesTautriple}
Let  $\SA$ be an abelian category with enough projective objects. Then every cotorsion torsion triple $(\SC, \ST, \SF)$ in $\SA$ is a $\tau$-cotorsion torsion triple.
\end{proposition}

\begin{proof}
Since $(\ST, \SF)$ is a torsion pair, it is enough to show that $(\SC, \ST)$ is a $\tau$-cotorsion pair.
The Condition $1$ of Definition \ref{WeakCotorsion}, holds trivially. Now let $P$ be a projective object in $\SA$ and consider the short exact sequence $0\lrt P \st{f}\lrt T \lrt C\lrt 0,$ in which $T\in\ST, C\in\SC$ and $f$ is a left $\ST$-approximation. It is exists by definition of a cotorsion pair.
To show the validity of Condition $2$ of Definition \ref{WeakCotorsion}, we just need to show that $T\in\SC\cap\ST$. This follows using the facts that $P, C \in \SC$ and $\SC$ is closed under extensions. The contravariantly finiteness of $\SC \cap \ST$ follows from Lemma \ref{ContraFinite}.
\end{proof}

\begin{proposition}\label{TauTriple}
Let $\SA$ be an abelian category with enough projective objects.
Let $(\SC, \ST, \SF)$ be a $\tau$-triple in $\SA$.
Then for every object $A\in\SA$ there exists an exact sequence \[ A\st{g}\lrt T\lrt C\lrt 0,\] where $T\in \ST$, $C\in\SC$ and $g$ is a left $\ST$-approximation.
\end{proposition}

\begin{proof}
Let $A$ be an arbitrary object of $\SA$. Let $P\st{\pi}\lrt A\lrt 0$ be an epimorphism from a projective object in $\SA$.
By $(3)$ of Definition \ref{WeakCotorsion}, we can construct the following  pushout diagram, where $C,\tilde{C}\in \SC\cap\ST$ and $f$ is a left $\ST$-approximation
\[\xymatrix{
   &P\ar[r]^{f}\ar[d]^{\pi}& C\ar[d]\ar[r]& \tilde{C}\ar@{=}[d]\ar[r]&0\\
 &A\ar[r]^{g}\ar[d]& T\ar[r]\ar[d]& \tilde{C}\ar[r] &0\\
& 0&0\\
}\]
We show that the second row in the above diagram is the desired exact sequence.
First, we note that $T\in\ST$, because $T\in\Fac(\ST)$ and $\ST$ is closed under quotients.
In order to complete the proof, it remains to show that $g: A\lrt T$ is a left $\ST$-approximation.
Let $h: A\lrt T'$ be a morphism in $\SA$ with $T'\in\ST$.
Since $f$ is a left $\ST$-approximation, there exists a morphism $l: C\lrt T'$ such that $lf=h\pi$.
Thus the pushout property implies the existence of the morphism $t: T\lrt T'$ such that the  diagram
 \[\begin{tikzcd}
P\ar{rr}{f}\dar{\pi}&& C\dar{}\ar[bend left]{ddr}{l}\\
A\ar{rr}{g}\ar[bend right]{drrr}{h}&& T\ar[dotted]{dr}{t}\\
&&& T'
\end{tikzcd}\]
is commutative.
In other words we have that $tg\pi = h\pi$.
Since $\pi : P \lrt A$ is an epimorphism, we obtain that $tg=h$.
Hence $g$ is a left $\ST$-approximation.
\end{proof}

\begin{corollary}\label{Cor-ff}
Let $\SA$ be an abelian category with enough projective objects. Let $(\SC, \ST, \SF)$ be a $\tau$-triple in $\SA$. Then $\ST$ is  a functorially finite torsion class of $\SA$.
\end{corollary}

\begin{proof}
The contravariantly finiteness of $\ST$ follows from the fact that it is a torsion class and its covariantly finiteness follows from the previous proposition.
\end{proof}

\begin{remark}\label{BZ}
In case $\SA=\mmod \La$, where $\La$ is an artin algebra, a generalization of the notion of a cotorsion pair, called a left weak cotorsion pair, is introduced and studied in \cite{BZ}. Based on Definition 0.2 of \cite{BZ} a pair $(\SC, \SD)$ of subcategories of $\mmod \La$ is a left weak cotorsion pair if
\begin{itemize}
  \item[$(1)$] $\Ext^1_{\La}(\SC, \SD)=0$.
  \item[$(2)$] For every $M \in \mmod \La$, there are  exact sequences
  \[M \st{f}{\lrt} D \lrt C \lrt 0,\] and \[0 \lrt D' \lrt C' \st{g}{\lrt} M \lrt 0\]
  such that $C, C' \in \SC$, $D, D' \in \SD$, $f$ is a left $\SD$-approximation of $M$ and $g$ is a right $\SC$-approximation of $M$.
\end{itemize}
A triple $(\SC, \ST, \SF)$ of full subcategories in $\mmod\La$ is called a left weak cotorsion torsion triple if $(\SC, \ST)$ is a left weak cotorsion pair and $(\ST, \SF)$ is a torsion pair.
\end{remark}

In the following theorem we show that in the module category of an artin algebra $\tau$-triples are exactly left weak cotorsion torsion triples.

\begin{theorem}\label{lwcotorsion}
Let $\SA=\mmod\La$, where $\La$ is  an artin algebra. Then  the triple $(\SC, \ST, \SF)$ of full subcategories in $\SA$ is a $\tau$-triple if and only if it is a left weak cotorsion torsion triple.
\end{theorem}

\begin{proof}
First, let $(\SC, \ST, \SF)$ be a $\tau$-triple.
Since $(\ST, \SF)$ is already a torsion pair, we just need to show that $(\SC, \ST)$ is a left weak cotorsion pair.
The Condition $(1)$ of the definition  of a left weak cotorsion pair follows by the first condition of the definition of $\tau$-cotorsion pair.
Also for $M\in \mmod\La$, by Proposition \ref{TauTriple}, there is an exact sequence \[M \st{f}{\lrt} T \lrt C \lrt 0\] where $T\in\ST$, $C\in\SC$ and $f$ is a left $\ST$-approximation.
Now for $M\in \mmod\La$, we construct a short exact sequence as in the Condition $(2)$ of the definition of left weak cotorsion pairs.
Let $P_1\st{\beta}\lrt P_0\lrt M\lrt 0$ be a projective
presentation of $M$.
By Condition $(2)$ of the definition of a $\tau$-cotorsion pair, for projective module $P_1$, there exists an exact sequence $P_1\st{\alpha}\lrt  T\lrt C\lrt 0$.
Consider the pushout diagram
\[\xymatrix{
& P_1\ar[d]^{\beta} \ar[r]^{\alpha} & T \ar[d]\ar[r] &C \ar@{=}[d]\ar[r] & 0 \\
& P_0\ar[r]\ar[d] & {C'}  \ar[r]\ar[d]^{g} &C \ar[r] & 0 \\
& M\ar[d]\ar@{=}[r]& M\ar[d]\\
&0&0}
\]
In view of the construction of pushouts in $\mmod\La$, we have the exact sequence
\[ P_1\st{\phi=\begin{bmatrix} \alpha\\ \beta \end{bmatrix}}\lrt T\oplus P_0\lrt C'\lrt 0,\]
and therefore  the short exact sequence
\[ 0\lrt \im\phi\st{\varphi: \begin{bmatrix} \iota\\ \kappa \end{bmatrix}}\lrt T\oplus P_0\lrt C'\lrt 0.\]
Now let $\pi: P_1\lrt \im\varphi\lrt 0$  and let $\psi:\im\varphi \lrt T'$ be a morphism with $T'\in\ST$.
We show that $\psi$ factors through $\varphi$.
Since $\alpha$ is a left $\ST$-approximation, then there exists a morphism $\gamma: T\lrt T'$ such that $\psi\pi=\gamma\alpha=\gamma\iota\pi$.
Thus $\psi=\gamma\iota$. By applying the functor $\Hom_\La(-,\ST)$ on the above short exact sequence and using the fact that $T\oplus P_0\in\SC$, we  get $C'\in\SC$.
Now since $\ST$ is a torsion class, it is closed under quotients, and so we have $\Ker g\in \ST$. Therefore the short exact sequence
\[0\lrt \Ker g\lrt C'\st{g}\lrt M\lrt 0\]
is the desired one.

Now we show the converse.
Let $(\SC, \ST, \SF)$ be a left weak cotorsion torsion triple.
Again since $(\ST, \SF)$ is a torsion pair, we just need to show that $(\SC, \ST)$ is a $\tau$-cotorsion pair.
By \cite[Lemma 4.1]{BZ}, we have the Condition $(1)$ of the definition of $\tau$-cotorsion pair.
To conclude the Condition $(2)$ of the definition of $\tau$-cotorsion pair, we note that $\ST$ is a functorially finite torsion class and by \cite{AS} there exists an exact sequence
\[ P \st{f}{\lrt} T \lrt C \lrt 0,\]
where $f$ is a left $\ST$-approximation, $T\in\ST\cap \SC$, $C\in\SC$.
Finally, the Condition $(3)$ of  the definition  of $\tau$-cotorsion pairs follows by \cite[Theorem 0.4]{BZ}.
\end{proof}

\section{$\tau$-tilting subcategories and $\tau$-triples}\label{Sec:Bijection}
In this section, we show that there is a bijection between the collection of all support $\tau$-tilting subcategories of $\SA$ and the collection of all $\tau$-cotorsion torsion triples in $\SA$.
In case we start with a tilting subcategory, this bijection specializes to the one introduced in \cite[Theorem 2.29]{BBOS}.
We prepare the ground with some preliminary results.

\begin{lemma}\label{IJY}
Let $\SA$ be an abelian category with enough projective objects.
Let $\ST$ be a support $\tau$-tilting subcategory of $\SA$. Then ${}^{\perp_1} \Fac(\ST) \cap \Fac(\ST)=\ST$.
\end{lemma}

\begin{proof}
Since  $\SA$ is an abelian category with  enough projectives $\Prj(\SA)$, we have $\SA\simeq\mmod\Prj(\SA)$, where $\mmod\Prj(\SA)$ is the category of finitely presented functors on $\Prj(\SA)$, see \cite[Corollaries 3.9 and 3.10]{Be1}.
Now the result  follows by the part $(ii)$ of  the proof of \cite[Proposition~5.3]{IJY}.
\end{proof}

\begin{proposition}\label{TauTiltingImpliesWeakCotorsion}
Let $\SA$ be an abelian category with enough projective objects.
Let $\ST$ be a  support $\tau$-tilting subcategory of  $\SA$.
Then
\[  ({}^{\perp_1} \Fac(\ST), \Fac(\ST))\]
is a $\tau$-cotorsion pair.
\end{proposition}

\begin{proof}
The first condition of Definition \ref{WeakCotorsion} holds trivially.
For the second condition, consider the exact sequence
\[P\st{f}\lrt T^0\lrt T^1\lrt 0\] where $f$ is a left $\ST$-approximation of $P$ and $T^0, T^1\in\ST$, which exists for every projective object $P$.
Now by Lemma \ref{IJY}, we observe that $T^0, T^1\in {}^{\perp_1} \Fac(\ST) \cap \Fac(\ST)$.
So it remains to show that $f$ is a left $\Fac(\ST)$-approximation of $P$.
To show this, let $X\in\Fac(\ST)$ and  let $g: P\lrt X$ be a morphism.
Consider an epimorphism $\pi:T\lrt X$ with $T\in\ST$.
Since $P$ is a projective object, there is a morphism $h: P\lrt T$ such that $\pi h=g$.
Now because $f$ is a left $\ST$-approximation, there is a morphism $t: T^0\lrt T$ such that $tf=h$.
Therefore a morphism $\pi t: T^0\lrt X$ exists such that $\pi t f=\pi h=g$.
Finally Lemma \ref{IJY}, implies that ${}^{\perp_1} \Fac(\ST) \cap \Fac(\ST)=\ST$ is a contravariantly finite subcategory of $\SA$.
\end{proof}

By \cite[Proposition 2.22]{BBOS} if $\ST$ is a tilting subcategory of $\SA$, then the pair $(\Fac(\ST), \ST^{\perp_0})$ is a torsion pair.
We use the same technique to prove the validity of the same result for $\tau$-tilting subcategories.

\begin{proposition}\label{tau-torsion pair}
Let $\SA$ be an abelian category with enough projective objects. If $\ST\subseteq\SA$ is a $\tau$-tilting subcategory, then $(\Fac(\ST), \ST^{\perp_0})$ is a torsion pair.
\end{proposition}

\begin{proof}
Let $X\in\Fac(\ST)$. Then there exists an epimorphism $\ST\lrt X\lrt 0$  with $T\in \ST$. The  exact sequence $0\lrt \Hom_\SA(X, Y)\lrt \Hom_\SA(T, Y)$ shows that $\Hom_\SA(X, Y)=0$ whenever $Y\in \ST^{\perp_0}$.

Now let $A\in\SA$ be an arbitrary object. Since $\ST$ is contravariantly finite subcategory, there exists a right $\ST$-approximation $\varphi: T\lrt A$. Consider the short exact sequence
\[0\lrt \im \varphi \st{f}\lrt A \lrt \Coker\varphi\lrt 0,\]
where $\im \varphi \in \Fac(\ST)$. By applying the functor $\Hom_\SA(\ST, -)$, we have a long exact exact sequence
\[0\lrt \Hom_\SA(-, \im\varphi)\vert_{\ST}\st{f_*} \lrt \Hom_\SA(-, A)\vert_{\ST} \lrt \Hom_\SA(-, \Coker \varphi)\vert_{\ST} \lrt \Ext^1_\SA(-, \im \varphi)\vert_{\ST}.\]
Now since $f$ is a right $\ST$-approximation, $f_*$ is an epimorphism. Also since $\ST$ is $\tau$-tilting and  $\im\varphi \in \Fac(\ST)$, $\Ext^1_\SA(-, \im \varphi)\vert_{\ST}=0$. So $\Hom_\SA(-, \Coker \varphi)\vert_{\ST}=0$ and hence $\Coker\varphi\in \ST^{\perp_0}$.
\end{proof}

As a consequence of this proposition we can record the following corollary, which is a generalization of \cite[Theorem 2.7]{AIR}, which proves the same result for the case when $\SA=\mmod\La$, the category of finitely generated modules over an artin algebra $\La$.

\begin{corollary}\label{FunctoriallyFiniteTorsion}
Let $\SA$ be an abelian category with enough projective objects. Let $\ST$ be a support $\tau$-tilting subcategory of $\SA$. Then $\Fac(\ST)$ is a functorially finite torsion class of $\SA$.
\end{corollary}

\begin{proof}
By Propositions \ref{TauTiltingImpliesWeakCotorsion} and \ref{tau-torsion pair},  $({}^{\perp_1} \Fac(\ST), \Fac(\ST), \ST^{\perp_0})$ is a $\tau$-cotorsion torsion triple. So the result follows from Corollary \ref{Cor-ff}.
\end{proof}

\begin{lemma}\label{WeakCotorsionImpliesTauTilting}
Let $\SA$ be an abelian category with enough projective objects.
Let $(\SC, \ST, \SF)$ be a $\tau$-cotorsion torsion triple in  $\SA$.
Then $\SC\cap\ST$ is a support $\tau$-tilting subcategory of $\SA$.
\end{lemma}

\begin{proof}
 First we observe that, by  the definition of $\tau$-cotorsion torsion triples, $\SC\cap\ST$ is a contravaraintly finite subcategory of $\SA$ and
\[ \Ext^1_\SA(\SC\cap\ST, \Fac(\SC\cap\ST))\subseteq \Ext^1_\SA(\SC, \Fac(\ST))=\Ext^1_\SA(\SC, \ST)=0.\]
Moreover, for every projective object $P$, there is an exact sequence
\[P\st{f}\lrt  T \lrt T' \lrt 0\]
where $T, T' \in \SC \cap \ST$ and $f$ is a left $\ST$-approximation.
To verify the last condition of support $\ST$-tilting subcategories, it is enough to note that the left $\ST$-approximation $f$ is also a left $\SC\cap\ST$-approximation.
\end{proof}

\begin{lemma}\label{FacIntersection}
Let $\SA$ be an abelian category with enough projective objects. Let $(\SC, \ST, \SF)$ be a $\tau$-cotorsion torsion triple in   $\SA$. Then
$\ST=\Fac(\SC\cap\ST)$.
\end{lemma}

\begin{proof}
Since $\ST$ is closed under factors, we observe that  $\Fac(\SC\cap\ST)\subseteq \ST$. Now let $T\in\ST$ and $P\st{\pi}\lrt T\lrt 0$ be an epimorphism with projective object $P$.
By Condition $(3)$ of Definition \ref{WeakCotorsion}, there is a left $\ST$-approximation $f: P\lrt T'$, where $T' \in\SC\cap\ST$.
So there exists an epimorphism $g: T' \lrt T$ such that $gf=\pi$. Hence $T \in \Fac(\SC\cap\ST)$.
\end{proof}

\begin{theorem}\label{Bijection}
Let $\SA$ be an abelian category with enough projective objects. Then there exits a bijection between the collection of all support $\tau$-tilting subcategories of $\SA$ and the collections of all $\tau$-cotorsion torsion triples in $\SA$. This bijection induces by the following maps
\[\lbrace \mbox{support $\tau$-tilting subcategories}\rbrace ~~~~\longleftrightarrow~~~~ \lbrace \mbox{$\tau$-cotorsion torsion triples}\rbrace\]
 $\  \  \  \  \  \ \ \ \ \ \ \ \ \ \ \ \ \ \ \ \ \ \ \ \ \  \ \ \ \ \ \ \ \ \ \ \ \ \ \ \ \ \  \ \ \ \ \ \ \ \  \ \ \ \ST~~~~\st{\Phi}\longrightarrow~~~~~({}^{\perp_1}{(\rm{Fac}(\ST))}, \rm{Fac}(\ST), \ST^{\perp_0})$
\[\  \ \ \SC\cap\ST~~\st{\Psi}\longleftarrow ~~~(\SC, \ST, \SF)\]
that are well defined and mutually inverse. Moreover, this bijection restricts to the bijection between the collection of all tilting subcategories of $\SA$ and the collections of all cotorsion torsion triples in $\SA$.
\end{theorem}

\begin{proof}
It follows from Propositions  \ref{TauTiltingImpliesWeakCotorsion} and \ref{tau-torsion pair} and Lemma \ref{WeakCotorsionImpliesTauTilting} that $\Phi$ and $\Psi$ are well defined. We show that they are mutually inverse.
Let $\ST$ be a support $\tau$-tilting subcategory.
We have that $\ST = {}^{\perp_1}(\Fac(\ST))\cap \Fac(\ST)$ by Lemma~\ref{IJY}.
Thus $\Psi\Phi\cong 1$.

Now let $(\SC, \ST, \SF)$ be a $\tau$-cotorsion torsion triple.
First we note that by Lemma \ref{FacIntersection}, $\ST=\Fac(\SC\cap\ST)$ and so $\SC={}^{\perp_1}{\ST}={}^{\perp_1} (\Fac (\SC\cap\ST))$.
Also we have $\SF=(\SC\cap\ST)^{\perp_0}$. Hence $\Phi\Psi\cong 1$ and so the first statement holds.

We now prove that this bijections restrict to the bijections between cotorsion torsion triples in $\SA$ and tilting subcategories in $\SA$.
Let $\ST$ be a tilting subcategory.
We show that the $\tau$-triple \[({}^{\perp_1}{(\rm{Fac}(\ST))}, \rm{Fac}(\ST), \ST^{\perp_0})\] is a cotorsion torsion triple.
To see this, it is enough to show that $({}^{\perp_1}{(\rm{Fac}(\ST))}, \rm{Fac}(\ST))$ is a cotorsion pair.
First we note that, since $\ST$ is a tilting subcategory of $\SA$, for every object $A\in\SA$ there is a short exact sequence
\[0\lrt A\lrt X_A\lrt Y_A\lrt 0\]
where $X_A\in\Fac(\ST)$ and $Y_A\in {}^{\perp_1}{(\rm{Fac}(\ST))}$.
Next we construct the second short exact sequence in the definition of cotorsion pair.
Let $A\in\SA$ and  let $0\lrt K\lrt P\lrt A\lrt 0$ be a short exact sequence in $\SA$, where $P$ is projective.
Let $0\lrt K\lrt X_K\lrt Y_K\lrt 0$ be a short exact sequence such that $X_K\in\Fac(\ST)$ and $Y_K\in{}^{\perp_1}{(\rm{Fac}(\ST))}$.
Consider the following pushout diagram
\[\xymatrix{
& & 0\ar[d]&0\ar[d]\\
  &0\ar[r] &K\ar[r]\ar[d]& P\ar[d]\ar[r]& A\ar@{=}[d]\ar[r]&0\\
 &0\ar[r] &X_K\ar[r]\ar[d]& U\ar[r]\ar[d]&A \ar[r] &0\\
&& Y_K\ar[d]\ar@{=}[r]&Y_K\ar[d]\\
&&0&0
}\]
Since $P, Y_K\in{}^{\perp_1}{(\rm{Fac}(\ST))}$ and ${}^{\perp_1}{(\rm{Fac}(\ST))}$ is closed under extensions, then $U\in{}^{\perp_1}{(\rm{Fac}(\ST))}$.
Hence the second row in the diagram is the desired short exact sequence.
Finally, it is clear  that $\Ext^1_\SA({}^{\perp_1}{(\rm{Fac}(\ST))}, \Fac(\ST))=0$.
So we have verified all of the conditions of a cotorsion pair.
\end{proof}

As a corollary of the above theorem, we can recover one of  the main results of \cite{BZ}. For this, we need some preparations. Let us begin by recalling the original
definition of a support $\tau$-tilting module \cite{AIR}. This definition is based on $\tau: \mmod \La \to \mmod \La$, the Auslander-Reiten translation in $\mmod \La$, see \cite[Chapter~IV]{ASS}.

\begin{definition}(see \cite[Definition~0.1]{AIR})\label{DefAIR}
Let $\Lambda$ be an artin algebra. A module $T$ in $\mmod\La$ is called $\tau$-rigid if $\Hom_{\La}(T, \tau T)=0$. It is called $\tau$-tilting if it is $\tau$-rigid and $|T|=|\La|$. A support $\tau$-tilting module $T$ in $\mmod \Lambda$ is a module $T$ that is a $\tau$-tilting module in $\mmod(\Lambda/\langle e\rangle)$, where $\langle e \rangle$ is the ideal generated by some idempotent $e \in \Lambda$.
\end{definition}

 As it is mentioned  in Remark  \ref{Jasso2.14}, by \cite[Proposition 2.14]{J}  $T$ is a support $\tau$-tilting module in $\mmod\La$ if and only if $\add(T)$ is a support $\tau$-tilting subcategory of $\mmod\La$.
We now show that every support $\tau$-tilting subcategory in $\mmod\La$ is of the form $\add(T)$ for some $\tau$-tilting module $T$ in $\mmod\La$.

\begin{proposition}\label{EquivalentDefinitions}
Let $\Lambda$ be an artin algebra. Then every support $\tau$-tilting subcategory $\ST$ of $\mmod \La$ is of the form $\add(T)$, where $T$ is a support $\tau$-tilting module in $\mmod \La$.
\end{proposition}

\begin{proof}
Let $\ST$ be a support $\tau$-tilting subcategory of $\mmod \La$.
By Proposition \ref{addgentautilt}, there exists a  support $\tau$-tilting module $T$ such that $\Fac(T)=\Fac(\ST)$.
We show that $\ST=\add(T)$.  First we note that by the construction of $T$ in Proposition \ref{addgentautilt}, we have $T\in\ST$. Therefore $\add(T)\in\ST$, since $T$ is additively closed subcategory of $\mmod\La$. Next, let $X\in\ST$, then $X\in\Fac(T)$.
By Proposition 2.5 of \cite{Z}, there exists a short exact sequence
\[0\lrt K'\lrt T'\lrt X\lrt 0\]
where $T'\in\add(T)$ and $K' \in\Fac(T)$.
Since $T$ is a $\tau$-tilting module, the above short exact sequence splits and  hence $X\in\add(T)$. Thus $\ST=\add(T)$.
 \end{proof}

\begin{corollary}(see \cite[Theorem 4.6]{BZ})
Let $\SA=\mmod\La$, where $\La$ is an artin algebra.
Then there is a bijection between the collection of all support $\tau$-tilting modules and the collection of all left weak cotorsion torsion triples.
\end{corollary}

\begin{proof}
This is a direct consequence of Theorem~\ref{lwcotorsion}, Theorem~\ref{Bijection} and Proposition~\ref{EquivalentDefinitions}.
\end{proof}

\section{Summary of dual results}\label{Sec:Dual}
In this section we collect the dual of our results in the previous sections. The proofs are similar, so we just list the results without their proofs.
Throughout this section we assume that $\SA$ is an abelian category with enough injective objects. For a subcategory $\SU$ of $\SA$, let $\Sub\SU$ be the full subcategory of $\SA$ consisting of all subobjects of finite direct sums of objects in  $\SU$.

We start by the definition of a $\tau^{-}$-tilting subcategory of $\SA$. Recall \cite[\S.2.2]{AIR} that a $\La$-module $M$, where $\La$ is an artin algebra, is called $\tau^{-}$-tilting if it is $\tau^{-}$-rigid, i.e. $\Hom_{\La}(\tau^{-}M, M)=0$, and $|M|=|\La|$. It follows from \cite[Proposition 5.6]{AS} that $M$ is $\tau^{-}$-rigid if and only if $\Ext^1_{\SA}(\Sub M,M)=0.$ Here $\Sub M$ means the subcategory of $\mmod \La$ consisting of all subobjects of $\add(M)$.
This motivates the following definition.

\begin{definition}\label{DefTauInverseTilting}
Let $\SA$ be an abelian category with enough injective objects. Let $\SU$ be an additive full subcategory of $\SA$. Then $\SU$ is called a  weak support $\tau^{-}$-tilting subcategory if
\begin{itemize}
  \item[$1.$] $\Ext^1_{\SA}(\Sub\SU,\SU)=0.$
  \item[$2.$] For every injective object $I$ in $\SA$, there exists an exact sequence
  \[0\lrt U^0\lrt U^1\st{g}\lrt I\]
  such that $U^0$ and $U^1$ are in $\SU$ and $g$ is a right $\SU$-approximation of $I$.
\end{itemize}
If furthermore $\SU$ is a covariantly finite subcategory of $\SA$, it is called a support $\tau^-$-tilting subcategory of $\SA$.
A support $\tau^-$-tilting subcategory $\SU$ of $\SA$ is called a $\tau^-$-tilting subcategory if
 the approximation $g: U^1 \lrt I$ is non-zero for every injective object $I$.
\end{definition}

\begin{definition}\label{TauInverseCotorsion}
Let $\SA$ be an abelian category with enough injective objects. A pair of full subcategories $(\SC, \SD)$ of $\SA$ is called a $\tau^-$-cotorsion pair if
\begin{itemize}
\item[$1.$] $\SD = \SC^{\perp_1}$.
\item[$2.$] For every injective object $I \in \SA$, there is an exact sequence
\[ 0\lrt D\lrt C\st{g}\lrt I,\]
 where $C \in \SC\cap\SD$, $D \in \SD$ and $g$ is a right $\SC$-approximation.
\item[$3.$] $\SC\cap\SD$ is a covariantly finite subcategory of $\SA$.
\end{itemize}
\end{definition}

\begin{definition}
A triple $(\ST, \SF, \SD)$ of full subcategories in $\SA$ is called a $\tau^-$-torsion cotorsion triple, or simply a $\tau^-$-triple, if $(\ST, \SF)$ is a torsion pair and $(\SF, \SD)$ is a  $\tau^-$-cotorsion pair.
\end{definition}

A triple $(\ST, \SF, \SD)$ of subcategories of $\SA$ is called a torsion cotorsion triple if $(\ST, \SF)$ is a torsion pair and $(\SF, \SD)$ is a cotorsion pair.

\begin{proposition}(Dual of Proposition \ref{CoTorImpliesTautriple})
Let $\SA$ be an abelian category with enough injective objects. Then  every  torsion cotorsion pair $(\ST, \SF, \SD)$ in $\SA$ is a $\tau^-$-triple.
\end{proposition}

\begin{proposition}(Dual of Proposition \ref{TauTriple})\label{TauInverseTriple}
Let $\SA$ be an abelian category with enough injective objects. Let $(\ST, \SF, \SD)$ be a $\tau^-$-triple in $\SA$. Then for every object $A\in\SA$, there exists an exact sequence \[  0\lrt D\lrt F\st{g}\lrt A,\] where $D\in \SD$, $F\in\SF$ and $g$ is a right $\SF$-approximation.
\end{proposition}

\begin{lemma}(Dual of  Lemma \ref{FacIntersection})\label{SubIntersection}
Let $\SA$ be an abelian category with enough injective objects. Let $(\ST, \SF, \SD)$ be a $\tau^-$-triple in   $\SA$. Then
$\SF=\Sub(\SF\cap\SD)$.
\end{lemma}

\begin{proposition}(Dual of  Propositions \ref{TauTiltingImpliesWeakCotorsion}, \ref{tau-torsion pair} and Lemma \ref{IJY})\label{TauInverseTiltingImpliesTauInverseCotorsion}
Let $\SA$ be an abelian category with enough injective objects. Let $\SU$ be a  support $\tau^-$-tilting subcategory of  $\SA$. Then
\[  (\Sub\SU, (\Sub\SU)^{\perp_1})\]
is a $\tau^-$-cotorsion pair and
\[({}^{\perp_0}{\SU}, \Sub\SU)\]
is a torsion pair in $\SA$. Moreover,
\[(\Sub\SU){}^{\perp_1}\cap \Sub\SU=\SU.\]
\end{proposition}

\begin{corollary}(Dual of Corollary \ref{FunctoriallyFiniteTorsion})\label{FunctoriallyFiniteTorsionFree}
Let $\SA$ be an abelian category with enough injectives. Let $\SU$ be a support $\tau^-$-tilting subcategory of $\SA$. Then $\Sub\SU$ is a functorially finite torsion free class of $\SA$.
\end{corollary}

\begin{lemma}(Dual of Lemma \ref{WeakCotorsionImpliesTauTilting})\label{dual of 5.6}
Let $\SA$ be an abelian category with enough injective objects. Let $(\ST, \SF, \SD)$ be a $\tau^-$-triple in  $\SA$. Then $\SF\cap\SD$ is a support $\tau^-$-tilting subcategory of $\SA$.
\end{lemma}

Let $\SU$ be an additively closed full subcategory $\SA$. By \cite[Subsection 2.3]{BBOS}, $\SU$ is called a cotilting subcategory if it satisfies the following conditions.
\begin{itemize}
\item[$(i)$] $\SU$ is a covariantly finite subcategory of $\SA$.
\item[$(ii)$] $\Ext^1_\SA(U_1, U_2)=0$, for all $U_1, U_2 \in \SU$.
\item[$(iii)$] Every object $U \in \SU$ has injective dimension at most $1$.
\item[$(iv)$] For every injective object $I$ in $\SA$, there exists a short exact sequence
\[0 \lrt U_1 {\lrt} U_0 \lrt I \lrt 0,\]
with $U^i \in \SU$.
\end{itemize}
If $\SU$ only satisfies the conditions $(ii)-(iv)$, it is called a weak cotilting subcategory of $\SA$.

\begin{theorem}(Dual of  Theorem \ref{Bijection})\label{Bijection-Dual}
Let $\SA$ be an abelian category with enough injective objects. Then there is a bijection
\[\lbrace \mbox{support $\tau^-$-tilting subcategories}\rbrace ~~~~\longleftrightarrow~~~~ \lbrace \mbox{$\tau^-$-torsion cotorsion triples}\rbrace\]
 $\  \  \  \  \  \ \ \ \ \ \ \ \ \ \ \ \ \ \ \ \ \ \ \ \ \  \ \ \ \ \ \ \ \ \ \ \ \ \ \ \ \ \  \ \ \ \ \ \ \ \  \ \ \ \SU~~~~\longrightarrow~~~~~({}^{\perp_0}{\SU}, \Sub\SU, (\Sub\SU)^{\perp_1})$
\[\  \ \ \ \ \ \ \  \SF\cap\SD~~\longleftarrow ~~~(\ST, \SF, \SD).\]
 This bijection restricts to a bijection between the collection of all cotilting subcategories of $\SA$ and the collections of all  torsion cotorsion  triples in $\SA$.
\end{theorem}

\begin{proposition}(Duall of Proposition \ref{EquivalentDefinitions})\label{DualEquivalentDefinitions}
Let $\Lambda$ be an artin algebra.
Then every support $\tau^{-}$-tilting subcategory $\ST$ of $\mmod \La$ is of the form $\add(T)$, where $T$ is a support $\tau^{-}$-tilting module in $\mmod \La$.
\end{proposition}

\section{$\tau$-$\tau^-$-quadruples}\label{Sec:Quadruples}

In this section, by combining the notions of $\tau$-triples and $\tau^{-}$-triples we are able to relate certain support $\tau$-tilting subcategories to certain support $\tau^{-}$-tilting subcategories of an abelian category with enough projective and enough injective objects. We show that this relation specializes to the dual of the dagger map introduced in \cite[Theorem~2.15]{AIR}, when we restrict $\SA$ to be $\mmod \La,$ the category of finitely presented modules over an artin algebra $\La$.

\begin{definition}\label{quadruple}
Let $\SA$ be an abelian category with enough projective and enough injective objects.
A quadruple $(\SC, \ST, \SF, \SD)$ of full additively closed subcategories of $\SA$ is called a $\tau$-$\tau^-$-quadruple if $(\SC, \ST, \SF)$ is a $\tau$-triple and $(\ST, \SF, \SD)$ is a $\tau^-$-triple.
\end{definition}

\begin{remark}\label{dagger map}
Let $(\SC, \ST, \SF, \SD)$ be a $\tau$-$\tau^-$-quadruple.
Then by Lemmas \ref{WeakCotorsionImpliesTauTilting} and \ref{dual of 5.6}, the map given by
\[\ddag: \SC \cap \ST \longmapsto \SF \cap \SD,\]
associates a $\tau$-tilting subcategory to a $\tau^-$-tilting subcategory.
\end{remark}

Let $\SA=\mmod\La$, where $\La$ is an artin algebra.
In \cite[Theorem  2.15]{AIR}, using the dual of a special map, which is called the dagger map and is denoted by $( - )^{\dagger}$, the authors constructed a bijection between the set of isomorphism classes of all support $\tau$-tilting $\La$-modules, denoted by ${\rm s}\tau\mbox{-}{\rm tilt}\La$ and the set of all isomorphism classes of support $\tau^{-}$-tilting $\La$-modules, denoted by ${\rm s}\tau^{-}\mbox{-}{\rm tilt}\La$.
This bijection is given by
$$(M,P) \mapsto (\tau M \oplus \nu P, \nu M_{\rm pr}),$$
where $\nu$ is the Nakayama functor and $M_{\rm pr}$ denotes the projective summand of $M$.
The next proposition shows that the map $\ddag$ defined in Remark \ref{dagger map} can be considered as a generalization of this map.
To this end we need the following easy lemma.

\begin{lemma}\label{FFTP}
Let $(\SC, \ST, \SF, \SD)$ be a $\tau$-$\tau^-$-quadruple.
Then both $\ST$ and $\SF$ are functorially finite subcategories of $\SA$. In this case, we say that $(\ST, \SF)$ is a functorially finite torsion pair of $\SA$.
\end{lemma}

\begin{proof}
Since $\ST$ is a torsion class, it is always contravariantly finite.
Moreover, Proposition \ref{TauTriple} implies that $\ST$ is a covariantly finite.
Hence it is functorially finite. By similar argument, $\SF$ is functorially finite.
Thus $(\ST, \SF)$ is a functorially finite torsion pair.
\end{proof}

Recall that two modules $M$ and $N$ are additively equivalent if $\add(M)=\add(N)$.

\begin{proposition}\label{daggermap}
Let $\SA=\mmod\La$.
Then the map associating $\SC \cap \ST$ to $\SF \cap \SD$ in every quadruple $(\SC, \ST, \SF, \SD)$, is exactly the dual of dagger map defined in \cite[Theorem 2.15]{AIR}, up to additive equivalences.
\end{proposition}

\begin{proof}
We show that there exists a bijection between the collection of all support $\tau$-tilting $\La$-modules and the collection of all $\tau$-$\tau^-$-quadruples in $\mmod \La$.
Let $M \in\mmod\La$ be a support $\tau$-tilting module, that is, $(M,P)$ is a $\tau$-tilting pair, for some projective module $P$.
Then $\ST=\add(M)$ is a $\tau$-tilting subcategory of $\mmod\La$ and so by Theorem \ref{Bijection},
\[({}^{\perp_1}\Fac(M), \Fac(M), M^{\perp_0})\]
is a $\tau$-triple.
On the other hand, in view of Theorem 2.15 of \cite{AIR}, $\tau M\oplus \nu P$ is a support $\tau^-$-tilting module.
More precisely, $(\tau M\oplus \nu P, \nu M_{{\rm pr}})$ is a $\tau^-$-tilting pair, where $M_{{\rm pr}}$ denotes the projective summand of $M$.
Hence  $\SU=\add(\tau M\oplus \nu P)$ is a $\tau^-$-tilting subcategory and so by Theorem \ref{Bijection-Dual},
\[({}^{\perp_0}(\tau M\oplus \nu P), \Sub(\tau M\oplus \nu P), (\Sub(\tau M\oplus \nu P))^{{\perp}_1})\]
is a $\tau^-$-torsion cotorsion triple.
Now ${}^{\perp_0}(\tau M\oplus \nu P)=\Fac(M)$ and $\Sub(\tau M\oplus \nu P)=M^{\perp_0}$ by \cite[Proposition 2.16.b]{AIR}.
In fact we get the following $\tau$-$\tau^{-}-$quadruple
\[({}^{\perp_1}\Fac(M), \Fac(M), \Sub(\tau M\oplus \nu P), (\Sub(\tau M\oplus \nu P))^{{\perp}_1}).\]
Conversely, assume that $(\SC, \ST, \SF, \SD)$ is a $\tau$-$\tau^-$-quadruple.
Since $\ST=\Fac(T)$ for some support $\tau$-tilting module $T$, Proposition~\ref{EquivalentDefinitions} implies that $\SC\cap\ST = \add (T)$ and $\SF \cap \SD = \add(\tau M\oplus \nu P)$. Hence the result follows.
\end{proof}

Let $\SA = \mmod \La$.
Then \cite[Theorem 1]{Sm} states that if $(\ST, \SF)$ is a torsion pair in $\SA$ then $\ST$ is functorially finite if and only if $\SF$ is functorially finite.
As we saw in the proof of the above theorem, we have that every $\tau$-triple $(\SC, \ST, \SF)$ can be completed to a $\tau$-$\tau^-$-quadruple.
In particular, this implies that the map $\ddag$ can be defined as a map from the set of all $\tau$-triples to the set of all $\tau^{-}$-triples.
This is not true in an arbitrary abelian category with enough projective objects, as is shown by the following example.

\begin{example}
Let $\SA=\RepR$.
Set
\begin{align*}
\ST=&\add(\lbrace  k_{[0, y)} \vert \ 0 < y\leq \infty\rbrace \cup \lbrace  k_{[x, y)} \vert \ 1 \leq x< y\leq \infty\rbrace),\\
\SF=&\add(\lbrace  k_{[x, y)} \vert \ 0< x <y \leq 1\rbrace).
\end{align*}
Then  by \cite[Example 2.5]{BBOS}, $(\ST, \SF)$ is a torsion pair.
Obviously $\ST$ is a functorially finite subcategory of $\SA$.
We show that $\SF$ is not  a contravariantly finite subcategory of $\SA$.
To see this, note that objects such as $K_{[0, b)}$, where $b<1$, have not a right $\SF$-approximation.
In fact, a right $\SF$-approximation of $K_{[0, b)}$ should be of the form $\theta: k_{[a, b)}\lrt k_{[0, b)}$, for some $0< a<b< 1$.
But non-zero morphisms such as $\eta: k_{[x, y)}\lrt k_{[0, b)}$, with $0\leq x<a$ and $b\leq y$, can not factor through $\theta$.
\end{example}

\section{Connection to silting and cosilting modules}\label{Sec:Connection}

This section is divided to two subsections and is devoted to the study the connections between support $\tau$- and support $\tau^{-}$-tilting subcategories and silting and cosilting theories in $\Mod R$, where $R$ is an associative unitary ring.  Based on these theories we are able to characterize all support  $\tau$- and  support $\tau^-$-tilting subcategories of $\Mod R$.

\subsection{Silting modules and support $\tau$-tilting subcategories}
Our aim in this subsection is to characterize all support $\tau$-tilting subcategories of $\Mod R$. We do this by providing a bijection between the equivalence classes of all support $\tau$-tilting subcategories of $\Mod R$ and the collection of all equivalent classes of the certain $R$-modules, the so-called finendo quasitilting $R$-modules. It is known that all silting modules are finendo quasitilting. As a result, it will be shown that $\Add(S)$ is a $\tau$-tilting subcategory of $\Mod R$, where $S$ is a silting $R$-module.

For a module $M$ in $\Mod R$, let $\Add(M)$ denote the class of all modules isomorphic to a direct summand of an arbitrary direct sum of copies of $M$.
We also let $\Gen(M)$ to be the subcategory of $\Mod R$ consisting of all $M$-generated modules, i.e. all modules isomorphic to an epimorphic images of modules in $\Add(M)$.
and $\Pres(M)$ to be the subcategory of $\Mod R$ consisting of all $M$-presented modules, i.e. all modules that admit an $\Add(M)$-presentation.
Recall that an $\Add(M)$-presentation of an $R$-module $X$ is an exact sequence
\[M_1 \lrt  M_0 \lrt X \lrt 0,\]
with $M_1$ and $M_0$ in $\Add(M)$.

Let $\sigma$ be a morphism in $\Prj(R) $. Let $\SD_{\sigma}$ be the class of all modules $M$ in $\Mod R$ such that the induced homomorphism $\Hom_R(\sigma, M)$ is surjective.

\begin{sdefinition}(see\cite[Definition 3.7]{AMV})
An $R$-module $S$ is called a partial silting module if there exists a projective presentation $\sigma$ of $S$ such that $\SD_{\sigma}$ contains $S$ and is a torsion class in $\Mod R$. $S$ is called a silting module if there is a projective presentation $\sigma$ of $S$ such that $\Gen(S)=\SD_{\sigma}.$
\end{sdefinition}

\begin{sremark}
By \cite[Remark 3.8]{AMV}, every silting module is a partial silting module, hence $\Gen(S)$ is a torsion class. Support $\tau$-tilting modules over a finite dimensional $k$-algebra are examples of silting modules \cite[Proposition 3.15]{AMV}.
\end{sremark}

\begin{sdefinition}(see \cite[Lemdef 3.1]{AMV})
An $R$-module $T$ is called quasitilting if $\Pres(T)=\Gen(T)$ and $T$ is Ext-projective in $\Gen(T)$.
\end{sdefinition}

Following proposition collects some of the basic properties of the  quasitilting modules. Their proofs can be found in  Lemdef 3.1, Lemma 3.3 and  Proposition 3.2 of \cite{AMV}. Recall that an $R$-modules $T$ is called finendo if it is finitely generated over its endomorphism ring.

\begin{sproposition}\label{AMV}
Let $T$ be a quasitilting  $R$-module.  Then the following statements hold true.
\begin{itemize}
\item[$1.$]
If $X\in\Gen(T)$ then, there exist a set $J$ and a short exact sequence \[0\lrt\Ker\pi\lrt T^{(J)}\st{\pi}\lrt X\lrt 0\]
such that $\Ker\pi\in\Gen(T)$ and $T^{(J)}$ is the coproduct of copies of $T$ indexed by $J$.
That is, $\Gen(T)$ is closed with respect to the kernels of epimorphisms.
\item[$2.$]
 $\Add(T)$ is the class of $\Ext$-projective modules in $\Gen(T)$.
 \item[$3.$] The following are equivalent.
 \begin{itemize}
 \item[$(i)$] $T$ is a finendo quasitilting $R$-module.
 \item[$(ii)$] $\Gen(T)$ is a torsion class and $T$ is a tilting $R/\ann(R)$-module.
 \item[$(iii)$] $T$ is an $\Ext$-projective module in $\Gen(T)$ and  there exists an exact sequence
 \[R\st{f}\lrt T_0\lrt T_1\lrt 0\]
 such that $T_0, T_1\in\Add(T)$ and $f$ is a left $\Gen(T)$-approximation.
 \end{itemize}
\end{itemize}
\end{sproposition}

\begin{sproposition}\label{FinendoImpliesSupport}
Let $T$ be a finendo quasitilting module in $\Mod R$. Then $\Add(T)$ is a support $\tau$-tilting subcategory of $\Mod R$.
\end{sproposition}

\begin{proof}
We show the validity of conditions of Definition \ref{DefTauTilting}.
Note that $\Fac(\Add(T))=\Gen(T)$.
Then $\Ext^1_R(\Add(T), \Fac(\Add(T)))=0$ follows from the statement $2$ of Proposition \ref{AMV}.

By Statement $3$ of Proposition \ref{AMV}, there exists an exact sequence
\[R\st{f}\lrt T_0\lrt T_1\lrt 0\]
where $T_0, T_1\in \Add(T)$ and $f$ is a left $\Gen(T)$-approximation. Since $T_0 \in \Add(T)$, $f$ is also a left $\Add(T)$-approximation. Now let $P$ be a projective $R$-module. There exists a set $J$ and an  epimorphism  $R^{(J)}\lrt P\lrt 0$ which is split. Consider the pushout diagram
\[\xymatrix{
  &R^{(J)}\ar[r]^{f^{(J)}} \ar[d]& T_0^{(J)}\ar[r]\ar[d]^{h}& T_1^{(J)}\ar[r]\ar@{=}[d]&0\\
 &P\ar[r]^{g} &\tilde{T}\ar[r]&  T_1^{(J)}\ar[r]&0\\
}\]
Since $h$ is a split epimorphism, we get $\tilde{T}\in \Add(T)$.
Also it follows easily from the universal property of the pushout diagrams that $g$ is a left $\Add(T)$-approximation of $P$.

In order to complete the proof, we just need to show that $\Add(T)$ is a contravariantly finte subcategory of $\Mod R$.
Let $M\in \Mod R$.
Since by Proposition \ref{AMV}.3, $\Gen(T)$ is a torsion class, it is a contravariantly finte subcategory of $\Mod R$.
Therefore, there is a monomorphism $0\lrt X\st{\imath}\lrt M $ such that $X\in \Gen(T)$.
By  Proposition \ref{AMV}.1, there exist a set $J$ and  a short exact sequence
\[0\lrt\Ker\pi \lrt T^{(J)}\st{\pi}\lrt X\lrt 0\]
such that $\Ker \pi\in\Gen(T)$.
We show that $T^{(J)}\st{\imath\pi}\lrt M$ is a right $\Add(T)$-approximation.
To do this, let $T'\in\Add(T)$ and $T'\st{\ell}\lrt M$ be a morphism.
Since $\imath$ is a right $\Gen(T)$-approximation, there is a morphism $T'\st{\jmath}\lrt X$ such that $\imath\jmath=\ell$.
By applying $\Hom_R(T', -)$ on the above short exact sequence and using the fact that $\Ker\pi\in\Gen(T)$, we conclude that $\jmath$ factors through $\pi$.
Hence $\ell$ factors through $\imath\pi$ and the result follows.
\end{proof}

As a direct consequence of the above proposition, we have the following.

\begin{scorollary}\label{SiltingImplySupTilting}
Let $S$ be a silting module in $\Mod R$. Then $\Add(S)$ is a support $\tau$-tilting subcategory of $\Mod R$.
\end{scorollary}

\begin{proof}
By \cite[Proposition 3.10]{AMV}, every silting $R$-module is a finendo quasitilting $R$-module. Now the result follows by the above proposition.
\end{proof}

We also have a kind of converse to the previous proposition.

\begin{sproposition}\label{SupportImpliesFinendo}
Let $\ST$ be a support $\tau$-tilting subcategory of $\Mod R$ such that $\ST=\Add(T)$, for some $R$-module $T$. Then $T$ is a finendo quasitilting $R$-module.
\end{sproposition}

\begin{proof}
Since $\ST$ is a support $\tau$-tilting subcategory, there exists a short exact sequence
\[ R\st{f}\lrt T^0\lrt T^1\lrt 0\]
where $T^0, T^1\in\ST$ and $f$ is a left $\ST$-approximation. We note that $f$  also is a left $\Fac(\ST)$-approximation, see for instance the proof of Proposition \ref{TauTiltingImpliesWeakCotorsion} for the proof of this fact. On the other hand, since $\Fac(\ST)=\Fac(\Add(T))=\Gen(T)$ and $\ST$ is a support $\tau$-tilting subcategory, $T$ is Ext-projective in $\Gen(T)$. Now Proposition \ref{AMV}.3  implies that $T$ is a finendo quasitilting module.
\end{proof}

The following result is a $\Mod$version of Proposition \ref{addgentautilt}.

\begin{sproposition}\label{SupportGenFinendo}
Let $\ST$ be a support $\tau$-tilting subcategory of $\Mod R$. Then there exists a finendo quasitilting module $T$ such that $\Fac(\ST)=\Gen(T)$.
\end{sproposition}

\begin{proof}
By Theorem \ref{Bijection},  $({}^{\perp_1}\Fac(\ST), \Fac(\ST), \ST^{\perp_0})$ is a $\tau$-triple.  Hence by Proposition \ref{TauTriple}, for every $M\in\Mod R$, there exists an exact sequence
\[M\st{\phi}\lrt B\lrt C\lrt 0\]
such that $\phi$ is a left $\Fac(\ST)$-approximation and $C\in {}^{\perp_1}\Fac(\ST)$, that is, $C$ is an Ext-projective in $\Fac(\ST)$. Hence by \cite[Theorem 3.4]{AMV}, we deduce that there exists a finendo quasitilting $R$-module $T$ such that $\Fac(\ST)=\Gen(T)$.
\end{proof}

\begin{sdefinition}
Let $\SA$ be an abelian category with enough projective objects. Let $\ST$ and $\ST'$ be two support $\tau$-titling subcategories of $\SA$. We say that $\ST$ and $\ST'$ are equivalent if  $\Fac(\ST)=\Fac(\ST')$.
\end{sdefinition}

We now have enough ingredients for the proof of our main theorem. Recall that, by \cite[Page 12]{AMV}, two quasitilting modules $T_1$ and $T_2$ are equivalent if $\Add(T_1)=\Add(T_2)$.

\begin{stheorem}\label{Bijection-Finendo}
There is a bijection between equivalence classes of support $\tau$-tilting subcategories of $\Mod R$ and equivalence classes of finendo quasitilting $R$-modules.
\end{stheorem}

\begin{proof}
The result follows by Propositions \ref{SupportGenFinendo} and \ref{FinendoImpliesSupport}.
\end{proof}

\subsection{Cosilting modules and support $\tau^-$-tilting subcategories}
Our aim in this subsection is to characterize all support $\tau^{-}$-tilting subcategories of $\Mod R$ in term of quasicotilting $R$-modules. It is known that all cosilting modules are quasicotilting. As a result, we show that $\Prod(C)$ is a $\tau^{-}$-tilting subcategory of $\Mod R$, where $C$ is a cosilting $R$-module, where $\Prod(M)$ denote the class of all modules isomorphic to an arbitrary direct product of copies of $M$.

For a module $M$ in $\Mod R$, let $\Cogen(M)$ be the subcategory of $\Mod R$ consisting of all $M$-cogenerated modules, i.e. all modules isomorphic to a submodule of modules in $\Prod(M)$ and let $\Copres(M)$ be the subcategory of $\Mod R$ consisting of all $M$-copresented modules, i.e. all modules that admit a $\Prod(M)$-copresentation.
Recall that a $\Prod(M)$-copresentation of an $R$-module $X$ is an exact sequence
\[0 \lrt  X \lrt M_0 \lrt M_1,\]
with $M_0$ and $M_1$ in $\Prod(M)$.

Let $\zeta$ be a morphism in $\Inj(R)$. Let $\SB_{\zeta}$ denote the class of all modules $M$ in $\Mod R$ such that the induced homomorphism $\Hom_R(M, \zeta)$ is surjective.

\begin{sdefinition}(see \cite[Definition 3.1]{BP})
An $R$-module $C$ is called a partial cosilting module if there is an injective copresentation $\zeta$ of $C$ such that $\SB_{\zeta}$ contains $C$ and the class $\SB_\zeta$ is closed under direct products.
Moreover, $C$ is called a cosilting module if there is an injective copresentation $\zeta$ of $C$ such that $\Cogen(C)=\SB_{\zeta}.$
\end{sdefinition}

By \cite[Remark 3.2]{BP}, every cosilting module is a partial cosilting module. In particular, for every cosilting module $C$, $\Cogen(C)$ is a torsion-free class.

\begin{sdefinition}(see \cite[Definition 2.1]{ZW})
An $R$-module $T$ is called a quasicotilting module if $\Cogen(T)=\Copres(T)$, $\Hom_R(-, T)$ preserves exactness of any short exact sequence in $\Cogen(T)$ and $T$ is $\Ext$-injective  in $\Cogen(T)$.
\end{sdefinition}

By \cite[Proposition 2.4]{ZW}, if $T$ is a quasicotilting $R$-module, then $\Cogen(T)$ is a torsion-free class.
Also \cite[Proposition 2.11]{ZW} implies that all quasicotilting $R$-modules are cofinendo.  Recall that $R$-module $T$ is cofinendo if and only if there exists a right $\Prod(T)$-approximation of an injective cogenerator $\Mod R$, see \cite[Proposition 1.6]{ATT}.

The following proposition collects some of the basic properties of the quasicotilting modules.
 The proofs can be found in   Lemma 3.1 and Theorem 3.2 of \cite{ZW}.

\begin{sproposition}\label{ZW}
Let $E$ be an injective cogenerator of $\Mod R$. Let $T$ be a quasicoilting $R$-module.
Then the following statements hold true.
\begin{itemize}
\item[$1.$]
If $X\in\Cogen(T)$ then there exist a set $J$ and a short exact sequence \[0\lrt X\st{\imath}\lrt T^J\lrt \Coker\imath\lrt 0\]
such that $\Coker\imath\in\Cogen(T)$ and $T^J$ is  the product of copies of $T$ indexed by $J$.
That is, $\Cogen(T)$ is closed with respect to the cokernels of monomorphisms.
\item[$2.$]
$\Prod(T)$ is the class of $\Ext$-injective module in $\Cogen(T)$.
 \item[$3.$] $T$ is quasicotilting if and only if $T$ is an $Ext$-injective  in $\Cogen(T)$ and
 there exists an exact sequence
 \[0\lrt T_0\lrt T_1\st{f}\lrt E \]
 such that $T_0, T_1\in\Prod(T)$ and   $f$ is a right  $\Cogen(T)$-approximation.
\end{itemize}
\end{sproposition}

In the following we show that every quasicotilting module induces a support $\tau^-$-tilting subcategory of $\Mod R$.
The proof is essentially dual to that of Proposition \ref{FinendoImpliesSupport}, but we include it for the sake of completeness.

\begin{sproposition}\label{CofinendoImpliesSupport}
Let $E$ be an injective cogenerator of $\Mod R$. Let $T$ be a quasicotilting  $R$-module. Then $\Prod(T)$ is a support $\tau^-$-tilting subcategory of $\Mod R$.
\end{sproposition}

\begin{proof}
We show the validity of the conditions of Definition \ref{DefTauInverseTilting}. For the first condition, note that $\Sub\Prod(T)=\Cogen(T)$.
The fact that
\[\Ext^1_R(\Sub\Prod(T), \Prod(T))=0\]
follows from Proposition~\ref{ZW}.2.
For the second condition, consider the exact sequence
\[0\lrt T_0\lrt T_1\st{f}\lrt E \]
of the statement $3$ of Proposition~\ref{ZW}, in which $T_0, T_1\in\Prod(T)$ and $f$ is a right $\Cogen(T)$-approximation of $E$.
Now let $I$ be an injective $R$-module.
There exists a set $J$ and a monomorphism $0\lrt I\lrt E^{J}$ which is a split morphism.
Consider the pullback diagram
\[\xymatrix{
 0\ar[r] &T_0^J\ar[r] \ar@{=}[d]&\tilde{T} \ar[r]^{g}\ar[d]^{h}&I\ar[d]\\
 0\ar[r]&T_0^J\ar[r] &T_1^J\ar[r]^{f^J}& E^J\\
}\]
Since $h$ is a split monomorphism, we get $\tilde{T}\in \Prod(T)$.
Moreover, it follows easily from the universal property of the pullback diagrams that $g$ is a right $\Prod(T)$-approximation of $I$.

To complete the proof it remains to show that $\Prod(T)$ is a covariantly finite subcategory of $\Mod R$.
Let $M\in\Mod R$.
Since $\Cogen(T)$ is a torsion-free class, it is a covariantly finite subcategory of $\Mod R$.
Therefore, there is an epimorphism $ M\st{\pi}\lrt X\lrt 0$ such that $X\in \Cogen(T)$.
By  Proposition \ref{ZW}.1, there exist a set $J$ and a short exact sequence
\[0\lrt X\st{\imath}\lrt T^{J}\lrt \Coker\imath\lrt 0\]
such that $\Coker \imath\in\Cogen(T)$.
We claim that $M \st{\imath \pi}\lrt T^{J}$ is a left $\Prod(T)$-approximation.
Indeed, let $f: M \lrt \Tilde{T}$ be a map with $\Tilde{T} \in \Prod(T)$.
Then $f$ factors through $X$ because $\Tilde{T} \in \Cogen(T)$.
Moreover, $\Cogen(T)$ is closed under cokernels of monomorphisms by Proposition~\ref{ZW}.1.
Hence we can lift the factorisation $f$ through $X$ using $\pi$ and conclude that $f$ factors through $\imath\pi$.
\end{proof}

As a direct consequence of the above proposition, we have the following.

\begin{scorollary}\label{Cosilting}
Let $C$ be a cosilting module in $\Mod R$  with respect to an injective copresentation $\zeta$.
Then $\Prod(C)$ is a support $\tau^-$-tilting subcategory of $\Mod R$.
\end{scorollary}

\begin{proof}
By \cite[Lemma 3.4]{BP}, $C$ is $\Ext$-injective in $\Cogen(T)$. Moreover, by \cite[Corollary 3.5]{BP}, we have $\Cogen(C)=\Copres(C)$.  Now \cite[Proposition 2.4]{BP} implies that every cosilting $R$-module $C$ is quasicotilting. Hence  the result follows by the above proposition.
\end{proof}

The following is a kind of converse to the previous proposition which is also the duall of Proposition \ref{SupportImpliesFinendo}.

\begin{sproposition}
Let $E$ be an injective cogenerator of $\Mod R$. Let $\SU$ be a support $\tau^-$-tilting subcategory of $\Mod R$ such that $\SU=\Prod(T)$, for some $R$-module $T$. Then $T$ is a  quasicotilting $R$-module.
\end{sproposition}

\begin{proof}
Since $\SU$ is a support $\tau^-$-tilting subcategory,  for $E$ there exists an exact sequence
\[0\lrt  U^0\lrt U^1\st{f}\lrt E\]
where $U^0, U^1\in\ST$ and $f$ is a right $\SU$-approximation. We note that $f$  also is a right $\Sub\SU$-approximation. On the other hand, since $\Sub\SU =\Sub \Prod(T) =\Cogen(T)$ and $\SU$ is a support $\tau^-$-tilting subcategory, $T$ is Ext-injective in $\Cogen(T)$. Now Proposition \ref{ZW}.3  implies that $T$ is a  quasicotilting module.
\end{proof}

The following result is a duall of Proposition \ref{SupportGenFinendo}.

\begin{sproposition}\label{TauInverseSupportCogenCofinendo}
Let $\SU$ be a support $\tau^-$-tilting subcategory of $\Mod R$. Then there exists a  quasicotilting module $T$ such that $\Sub \SU=\Cogen(T)$.
\end{sproposition}

\begin{proof}
By Theorem \ref{Bijection-Dual},  $({}^{\perp_0}{\SU}, \Sub\SU, (\Sub\SU)^{\perp_1})$ is a $\tau^-$-triple.  Hence by Proposition \ref{TauInverseTriple}, for every $M\in\Mod R$, there exists an exact sequence
\[0\lrt B\lrt C\st{\phi}\lrt M\]
such that $\phi$ is a right $\Sub \SU$-approximation and $B\in (\Sub\SU)^{\perp_1}$, that is, $B$ is an $\Ext$-injective module in $\Sub\SU$. Hence by \cite[Theorem 3.5]{ZW}, we deduce that there exists a  quasicotilting $R$-module $T$ such that $\Sub \SU=\Cogen(T)$.
\end{proof}

\begin{sdefinition}
Let $\SA$ be an abelian category with enough injective objects. Let $\SU$ and $\SU'$ be two support $\tau^-$-titling subcategories of $\SA$. We say that $\SU$ and $\SU'$ are equivalent if  $\Sub \SU=\Sub \SU'.$
\end{sdefinition}

Now we can state the main theorem of this subsection which is the dual of Theorem \ref{Bijection-Finendo}.
Recall that, by \cite[Page 12]{ZW}, two quasicotilting modules $T_1$ and $T_2$ are equivalent if $\Prod(T_1)=\Prod(T_2)$.

\begin{stheorem}\label{Bijection-Cofinendo}
There is a bijection between equivalence classes of support $\tau^-$-tilting subcategories of $\Mod R$ and equivalence classes of  quasicotilting $R$-modules.
\end{stheorem}

\begin{proof}
The result follows by Propositions \ref{TauInverseSupportCogenCofinendo} and \ref{CofinendoImpliesSupport}.
\end{proof}

\section{Applications to quiver representations}\label{Sec:Quivers}
This section is devoted to produce support $\tau$-tilting and $n$-tilting subcategories of category of representation of quivers, where $n$ is a non-negative integer.
We divide the section in four subsections.
In the first subsection we recall some known definitions and properties of the category $\Rep(Q, \SA)$ of representations of a finite acyclic quiver $Q$ over an abelian category $\SA$ with enough projective objects.
In the second subsection we produce support $\tau$-tilting subcategories in $\Rep(Q, \SA)$ from certain $\tau$-tilting subcategories of $\SA$.
In the third we construct (co)silting modules in $\Mod RQ$ from (co)silting modules in $\Mod R$.
Finally, in the last subsection, we use similar techniques to produce $(n+1)$-tilting subcategories of $\Rep(Q,\SA)$ from $n$-tilting subcategories of $\SA$.

\subsection{Notions on quiver representations}
Let $\SA$ be an abelian category with enough projective objects $\Prj(\SA)$ and  $Q=(Q_0, Q_1)$ be a finite acyclic quiver with vertex set $Q_0$ and arrow set $Q_1$.
An arrow $\alpha\in Q_1$ of sourse $i=s(\alpha)$ and target $j=t(\alpha)$    is usually denoted by $\alpha: i\rt j$.
We denote by $\Rep(Q, \SA)$ the category of representations of $Q$ in $\SA$. An object $X$ in $\Rep(Q, \SA)$  is defined by the following data:
\begin{itemize}
\item[$1.$] To each vertex $i\in Q_0$ is associated an  object $X_i$ in $\SA$.
\item[$2.$] To each arrow $\alpha: i\rt j$ in $Q_1$ is associated a morphism $X_\alpha: X_i\lrt X_j$.
\end{itemize}
A morphism $\varphi: X \lrt Y$ in $\Rep(Q,\SA)$ is a family $\lbrace \varphi_i: X_i\lrt Y_i \rbrace_{ i\in Q_0}$ of morphisms in $\SA$ such that for each arrow $\alpha: i\lrt j$ in $Q_1$, the diagram
\[\xymatrix{
  &X_i\ar[r]^{\varphi_i} \ar[d]^{X_\alpha}&Y_i\ar[d]^{Y_\alpha}\\
 &X_j\ar[r]^{\varphi_j} &Y_j\\
}\]
is commutative.

The category $\Rep(Q, \SA)$ is an abelian category.
Kernels, cokernels, and images in $\Rep(Q,\SA)$ are computed vertex-wise in $\SA$.
In fact, a sequence $X \lrt Y \lrt Z$ in $\Rep(Q,\SA)$ is exact if and only if for every vertex $i \in Q_0$, the sequence $X_i\lrt Y_i\lrt Z_i$ is exact in $\SA$.

For each vertex $i\in Q_0$, there exists the evaluation functor
\[e_i: \Rep(Q, \SA)\lrt \SA\]
\[ \ \ \ \ \ \ \ \ \ \  \ \ \ \ \ \ X\mapsto X_i\]
which sends each representation $X\in\Rep(Q, \SA)$ to the object $X_i\in\SA$ at vertex $i$.
It is clear that the evaluation functor $e_i$ is exact and moreover it has an exact left and also an exact right adjoint, which will be denoted by $e_i^\lambda$ and $e_i^\rho$, respectively.
Let us recall the constructions of $e_i^\lambda,e_i^\rho: \SA\lrt \Rep(Q, \SA)$ more explicitly, \emph{cf.} \cite{HJ}.

Let $A\in\SA$.
Then $e_i^\lambda(A)_j= \bigoplus_{Q(i, j)} A$, where $Q(i, j)$ denotes the set of paths starting in $i$ and ending in $j$.
The morphisms are natural inclusions, that is, for any arrow $\alpha: j\lrt k$, we set $e_i^\lambda(A)_\alpha: \bigoplus_{Q(i, j)} A\lrt \bigoplus_{Q(i, k)} A$.

The right adjoint $e_i^\rho$ is defined dually.
Let $A\in\SA$.
Then $e_i^\rho(A)_j=\bigoplus_{Q(j, i)} A$.
The morphisms are natural projections.
Moreover, the functor $e_i^\rho$ has a right adjoint, which will be denoted by $Re_i^\rho$.

It is proved that the sets
\[\lbrace e_i^\lambda (P): \  i\in Q_0, P\in\Prj(\SA)\rbrace \ \ {\rm and} \ \ \lbrace e_i^{\rho}(I): \ i\in Q_0, I\in\Inj(\SA) \rbrace\]
are sets of projective generators and injective cogenerators for the category $\Rep(Q, \SA)$, respectively.
For details of the proofs see e.g. \cite{EE} and \cite{EER}.

\subsection{Constructing $\tau$-tilting subcategories of $\Rep(Q, \SA)$}
Our aim in this subsection is to provide a systematic technique to construct, starting from a certain $\tau$-tilting subcategory $\ST$ of an abelian category $\SA$, a new $\tau$-tilting subcategory in the category of representation of a finite and acyclic quiver in $\SA$.
For the proof of the main result of this subsection, we need the following lemma.
Although it seems that it is known to the experts, we could not find a reference.
So we provide a proof for the sake of completeness.

\begin{slemma} \label{Keller}
Let $\SA$ be an abelian category with enough projective objects and $Q=(Q_0, Q_1)$ be a finite and acyclic quiver.
Then for every $X, Y \in \Rep(Q, \SA)$ there exists the long exact sequence
\begin{align*}
0 \lrt & \ \Hom_{\SR}(X, Y)\lrt \bigoplus_{r\in Q_0} \Hom_\SA(X_r, Y_r) \st{\varphi}\lrt \bigoplus_{\alpha: r\rt l} \Hom_\SA(X_r, Y_l)  \\
  \lrt & \ \Ext^1_{\SR}(X, Y) \ \lrt \ \bigoplus_{r\in Q_0}\Ext^1_\SA(X_r, Y_r) \ \lrt \bigoplus_{\alpha: r\rt l}\Ext^1_\SA(X_r, Y_l) \\
  \lrt & \ \Ext^2_{\SR}(X, Y) \ \lrt \cdots,
\end{align*}
where here and throughout we set $\SR:=\Rep(Q, \SA)$.
\end{slemma}

\begin{proof}
Let $X \in \SR:=\Rep(Q, \SA)$.
By \cite[Lemma 3.5]{BBOS}, there exists a short exact sequence
\[0\lrt \bigoplus_{\alpha:r\rt l} e_l^{\la}(X_r)\lrt \bigoplus _{r\in Q_0} e_r^{\la}(X_r)\lrt X\lrt 0\]
which is natural in $X$.
By applying the functor $\Hom_{\SR}(-,Y)$ to this sequence, we get the following long exact sequence
\begin{align*}
0 \lrt & \ \Hom_{\SR}(X, Y) \lrt \bigoplus_{r\in Q_0}\Hom_{\SR}( e_r^{\la}(X_r), Y)\lrt \bigoplus_{\alpha:r\rt l}\Hom_{\SR}(e_l^{\la}(X_r), Y)\\
  \lrt & \ \Ext^1_{\SR}(X, Y) \ \ \lrt\bigoplus_{r\in Q_0}\Ext^1_\SR( e_r^{\la}(X_r), Y) \ \lrt \bigoplus_{\alpha:r\rt l}\Ext^1_\SR(e_l^{\la}(X_r), Y)\\
  \lrt & \ \Ext^2_{\SR}(X, Y) \ \ \lrt \cdots.
\end{align*}
Now the result follows in view of the adjoint pair $(e_i^{\la}, e_i)$ and using the fact that the adjunction between $e_i^{\la}$ and $e_i$ extends to $\Ext^t$, for all $t \geq i$, see \cite[Proposition 5.2]{HJ}.
\end{proof}

Now we can state and prove the main result of this part.
\begin{stheorem}\label{ProduceTauTilting}
Let $\SA$ be an abelian category with enough projective objects and $Q=(Q_0, Q_1)$ be a finite and acyclic quiver.
Let $\ST$ be a support $\tau$-tilting subcategory of $\SA$ such that $\Fac(\ST)$ is closed with respect to the kernels of epimorphisms.
Then
\[\mathbb{T} =\add\lbrace e_i^{\rho}(T) \vert ~ i\in Q_0, T\in\ST\rbrace\]
is a support $\tau$-tilting subcategory of $\Rep(Q, \SA)$.
\end{stheorem}

\begin{proof}
We show the validity of conditions of Definition \ref{DefTauTilting}.
For the first condition, let $i, j\in Q_0$ and $T\in \ST$.
We show that
\[\Ext^1_{\SR}(e_i^{\rho}(T), \Fac(e_j^{\rho}(T))=0.\]
Set $X:=e_i^{\rho}(T)$ and pick $Y\in\Fac(e_j^{\rho}(T))$.
Then for every $r\in Q_0$, $X_r$ is the sum of some finite copies of $T$, maybe zero, and $Y_r$ is in $\Fac(T)$.
Moreover, for every $\alpha: r\rt l\in Q_1$, $Y_\alpha: Y_r\lrt Y_l$ is an epimorphism.
Hence since by assumption $\Fac(\ST)$ is closed with respect to the kernels of epimorphisms, we deduce that $\Ker Y_\alpha \in \Fac(T)$.
This in particular implies that for every $T \in \ST$, the induced morphism
\[\Hom_\SA(T, Y_r) \lrt \Hom_\SA(T, Y_l)\]
is an epimorphism.

Hence, in the exact sequence
\[ \bigoplus_{r\in Q_0} \Hom_\SA(X_r, Y_r) \st{\varphi}\lrt \bigoplus_{\alpha: r\rt l} \Hom_\SA(X_r, Y_l) \lrt \\  \Ext^1_{\SR}(X, Y)\lrt \bigoplus_{r\in Q_0}\Ext^1_\SA(X_r, Y_r),\] of the above lemma, we deduce that $\varphi$ is an epimorphism.
So to show the result, it is enough to show that
\[\bigoplus_{r\in Q_0}\Ext^1_\SA(X_r, Y_r)=0.\]
This follows from the fact $\Ext^1_\SA(\ST, \Fac(\ST))$, because $\ST$ is a support $\tau$-tilting subcategory of $\SA$ and the fact that
\[\bigoplus_{r\in Q_0}\Ext^1_\SA(X_r, Y_r)\subseteq\Ext^1_\SA(\ST, \Fac(\ST)).\]

Now we show the validity of the second condition.
It is enough to show it only for the projective generators of $\Rep(Q, \SA)$, i.e. for representations of the form
$e_i^\la(P)$, where $P$ is a projective object in $\SA$.
Let $\lbrace \rho_1, \cdots, \rho_k \rbrace$ be the set of all longest paths in $Q$ starting from $i$.
Since $\ST$ is a support $\tau$-tilting subcategory of $\SA$, for $P$ there exists an exact sequence $P\st{f}\lrt T^0\st{g}\lrt T^1\lrt 0$, such that $T^0, T^1\in\ST$ and $f$ is a left $\ST$-approximation of $P$.
Since $\Fac(\ST)$ is closed with respect to the kernels of epimorphisms, $\Ker g \in \Fac(\ST)$ and so the induced short exact sequence
\[0\lrt \Ker g\lrt T^0\lrt T^1\lrt 0\]
splits and so $\Ker g\in\ST$.
Take the exact sequence
\[e^{\la}_{i}(P)\st{\psi}\lrt  \bigoplus_{q=1}^ k e^\rho_{t(\rho_q)}(T^0)\lrt  \bigoplus_{q=1}^k e^\rho_{t(\rho_q)}(T^1)\oplus  \bigoplus_{\alpha \in \mathfrak{I} }e^\rho_{s(\alpha)}(\Ker g)\lrt 0\]
where $\mathfrak{I}\subset Q_1$ is the set of arrows $\alpha$ of $Q$ such that $\alpha$ is not part of any path in the set $\lbrace \rho_1, \cdots, \rho_k \rbrace$ but there is a path in that set passing through $t(\alpha)$.
An easy verification shows that this is the desired sequence for $e^\la_i(P)$.
In particular, $\psi$ is a left $\mathbb{T}$-approximation of $e^{\la}_{i}(P)$.

In order to complete the proof, we have to show that $\mathbb{T}$ is a contravariantly finite subcategory of $\Rep(Q, \SA)$.
Let $X \in\Rep(Q, \SA)$. For each $i$, consider a right $\ST$-approximation $\pi^i: T^i \lrt Re_i^\rho(X)$, where $Re_i^\rho$ is the right adjoint of $e_i^\rho$.
Following the same argument as in \cite[Proposition 3.9]{BBOS} one can show that \[\bigoplus_{i\in Q_0} e_i^\rho(T^i) \lrt X\] is a right $\mathbb{T}$-approximation of $X$.
\end{proof}

Following examples provide situations where a support $\tau$-tilting subcategory of $\SA$ has the property that $\Fac(\ST)$ is closed with respect to the kernels of epimorphisms.

\begin{sexample}
Let $\La$ be an artin algebra and $S$ be a simple injective object in $\mmod\La$.
Then $\add(S)$ is a  support $\tau$-tilting subcategory such that $\Fac(S)=\add(S)$ is closed under kernels of epimorphisms.
\end{sexample}

\begin{sexample}
Let $A$ be a finite dimensional algebra, $e$ be an idempotent of $A$ and $B=A/AeA$.
Let $Q=(Q_0, Q_1)$ be a finite and acyclic quiver.
By \cite[Proposition 3.9]{BBOS}, we have $\ST=\add (\lbrace e_i^\rho(B) \vert i\in Q_0\rbrace)$  is a tilting subcategory in $\Rep(Q, \mmod B)$.
On the other hand, it is obvious that  $\add(B)$ is a support $\tau$-tilting subcategory of $\mmod A$ and $\Fac(\add(B))=\mmod B$ is closed under kernels of epimorphisms.
So by Theorem \ref{ProduceTauTilting}, $\ST$ is a support $\tau$-tilting of $\Rep(Q, \mmod A)$.
\end{sexample}

We end this subsection by the following example which is also an application of the Theorem \ref{ProducingTauTiltingFromTiltingRep}.

\begin{sexample}
Let $Q$ be a finite and acyclic quiver and $Q'$ be a full subquiver of $Q$.
Then it is immediate that $\Rep(Q', \SA)$ is a wide and functorially finite torsion class of $\Rep(Q, A)$. By \cite[Proposition 3.9]{BBOS}, $\ST=\add(\lbrace e_i^\rho (P)\ \vert i\in Q'_0,  P\in \Prj(\SA)\rbrace)$ is a tilting subcategory of $\Rep(Q', \SA)$.
So by Theorem \ref{ProducingTauTiltingFromTiltingRep}, $\ST$ is a support $\tau$-tilting subcategory of $\Rep(Q, \SA)$.
\end{sexample}

\subsection{(Co)silting objects in $\Mod RQ$}
Let $R$ be an associative ring with unity and $Q$ be a finite and acyclic quiver. In this subsection we construct silting, resp. cosilting, objects in the category of representations of $Q$ in $\Mod R$, $\Rep(Q, \Mod R)$, from silting, resp. cosilting, modules in $\Mod R$.
Note that $\Rep(Q, \Mod R)$ is equivalent to the $\Mod RQ$, where $RQ$ denotes the path algebra of $Q$ over $R$.
So by a silting, resp. cosilting, object in $\Rep(Q, \Mod R)$ we mean a silting, resp. cosilting, module in $\Mod RQ$.

\begin{stheorem}
Let $Q=(Q_0, Q_1)$ be a finite and acyclic quiver.
\begin{itemize}
\item[$(i)$]
Let $S$ be a silting module in $\Mod R$. Let $i\in Q_0$ be an arbitrary vertex of $Q$. Then $e_i^\la(S)$ is a silting object in $\Rep(Q, \Mod R)$.
\item[$(ii)$]
Let $C$ be a cosilting module in $\Mod R$. Let $i\in Q_0$ be an arbitrary vertex of $Q$. Then $e_i^\rho(C)$ is a cosilting object in $\Rep(Q, \Mod R)$.
\end{itemize}
\end{stheorem}

\begin{proof}
$(i)$ Let $S$ be a silting $R$-module. By definition, there exists a projective presentation
\[P_1\st{\sigma} \lrt P_0\lrt S\lrt 0\]
of $S$ such that $\SD_\sigma=\Gen(S)$. By applying the exact functor $e_i^\la$ on the projective presentation of $S$ and using the fact that $e_i^\la$ is an exact functor that preserves projectives, we get the projective presentation
\[e_i^\la(P_1)\st{e_i^\la(\sigma)}\lrt e_i^\la(P_0)\lrt e_i^\la(S)\lrt 0\]
of $e_i^\la(S)$. To complete the proof, we show that $\SD_{e_i^\la(\sigma)}=\Gen(e_i^\la(S))$.

Let $X \in \SD_{e_i^\la(\sigma)}$. So there exists an epimorphism
\[\Hom_\SR(e_i^\la(P_0), X)\lrt \Hom_\SR(e_i^\la(P_1), X)\lrt 0,\]
where $\SR$ means $\Rep(Q, \Mod R)$. The adjoint pair $(e_i^\la, e_i)$ induces the epimorphism
\[\Hom_R(P_0, X_i)\lrt \Hom_R(P_1, X_i)\lrt 0.\]
This, in turn, implies that $X_i \in \SD_\sigma=\Gen(S)$. Therefore, $e_i^\la(X_i) \in e_i^\la(\Gen(S))$. On the other hand, by \cite[Lemma 3.5]{BBOS}, there exists an epimorphism $\bigoplus_{i\in Q_0} e_i^\la(X_i)\lrt X\lrt  0$ which shows that $X\in e_i^\la(\Gen(S))$. But it follows directly from the definition of $e_i^\la$ that $e_i^\la(\Gen(S)) = \Gen(e_i^\la(S))$. Thus $\SD_{e_i^\la(\sigma)}\subseteq \Gen(e_i^\la(S))$.

To see the reverse inclusion, let $X\in e_i^\la(\Gen(S))$. So $X=e_i^\la(U)$ such that $U\in\Gen(S)$. Since $\Gen(S)=\SD_\sigma$, we have $U\in\SD_\sigma$. Therefore, there exists an epimorphism
\[\Hom_R(P_0, U)\lrt \Hom_R(P_1, U)\lrt 0.\]
By the using of adjoint properties of adjoint pair $(e_i^\la, e_i)$, we have an epimorphism
\[\Hom_\SR(e_i^\la(P_0), X)\lrt \Hom_\SR(e_i^\la(P_1), X)\lrt 0\]
which shows that $X\in\SD_{e_i^\la(\sigma)}$. So we show that $e_i^\la(\Gen(S))= \Gen(e_i^\la(S)) \subseteq \SD_{e_i^\la(\sigma)}$. Hence the proof is complete.

$(ii)$ The proof is just dual of the proof of part $(i)$, so we skip the proof.
\end{proof}

The following result provides a partial converse to the above theorem. Recall that a vertex $i\in Q_0$ is called a source, resp. a sink, of $Q$ if there is no arrows $\alpha \in Q_1$ such that $t(\alpha)=i$, resp. $s(\alpha)=i$.

\begin{stheorem}
Let $Q=(Q_0, Q_1)$ be a finite and acyclic quiver.
\begin{itemize}
\item[$(i)$]
Let $X$ be a silting  object in $\Rep(Q, \Mod R)$. Then $e_i(X)$ is a silting module in $\Mod R$, provided $i\in Q_0$ is a source of $Q$.
\item[$(ii)$]
Let $Y$ be a cosilting object  in $\Rep(Q, \Mod R)$. Then $e_i(Y)$ is a cosilting  module in $\Mod R$, provided $i\in Q_0$ is a sink of $Q$.
\end{itemize}
\end{stheorem}

\begin{proof}
$(i)$ Since $X$ is a silting object in $\Rep(Q, \Mod R)$, there exists a projective presentation
\[P^1\st{\sigma}\lrt P^0\lrt X\lrt 0\]
of $X$ such that $\SD_\sigma=\Gen(X)$. By applying the exact functor $e_i$ on the this exact sequence we get the exact sequence
\[P^1_i\st{\sigma_i}\lrt P^0_i\lrt X_i\lrt 0,\]
which is a projective presentation of $e_i(X)=X_i$.

In order to complete the proof, we have to show that $\SD_{\sigma_i}=\Gen(X_i)$.  First let $M\in\SD_{\sigma_i}$. Then there exists an epimorphism
\[\Hom_R( P^0_i, M)\lrt \Hom_R( P^1_i, M)\lrt 0.\]
By using the   adjoint pair $(e_i, e_i^\rho)$, we get the epimorphism
\[\Hom_\SR(P^0, e_i^\rho(M))\lrt \Hom_\SR(P^1, e_i^\rho(M)\lrt 0\]
where $\SR$ means $\Rep(Q, \Mod R)$.  So $e_i^\rho(M)\in \SD_\sigma=\Gen(X)$. Therefore $M=e_i e_i^\rho(M)\in e_i(\Gen(X))$. But $e_i(\Gen(X))=\Gen(e_i(X))=\Gen(X_i)$. Hence $\SD_{\sigma_i}\subseteq\Gen(X_i)$.

For the reverse inclusion, let $Y_i\in\Gen(X_i)$.  We define $Y\in \Rep(Q, \Mod R)$  by setting $Y_i$ at source vertex $i$, and $0$ elsewhere. Since $i$ is a source, it follows easily that $Y \in \Gen(X)$. Therefore $Y \in \SD_\sigma$ which implies that $Y_i \in \SD_{\sigma_i}$. Hence we have the equality $\SD_{\sigma_i}=\Gen(X_i)$ and the proof is complete.

$(ii)$ The proof is just dual of the proof of part $(i)$. So we skip the proof.
\end{proof}

\subsection{Higher tilting subcategories of $\Rep(Q, \SA)$}
Let $\SA$ be an abelian category with enough projective objects.
Let $\Prj(\SA)$ denote the subcategory of $\SA$ consisting of all projective objects.
In \cite[Proposition 3.9]{BBOS} it is shown that
\[\ST'=\add(\lbrace e_i^{\rho}(P) \vert ~ i\in Q_0,  P \in \Prj(\SA) \rbrace)\]
is a ($1$-)tilting subcategory of $\Rep(Q, \SA)$, where $Q$ is a finite and acyclic quiver.

Now if we interpret $\Prj(\SA)$ as a $0$-tilting subcategory of $\SA$, then by the above result, starting from a $0$-tilting subcategory of $\SA$ we get a $1$-tilting subcategory of $\Rep(Q,\SA)$.
In our next and last result we provide a higher version of this result by constructing an $(n+1)$-tilting subcategory in $\Rep(Q,\SA)$ starting from an $n$-tilting subcategory of $\SA$.

\begin{stheorem}\label{ThHigherTilt}
Let $\SA$ be an abelian category with enough projective objects.
Let $n$ be a non-negative integer. Let $Q$ be a finite and acyclic quiver.
For an $n$-tilting subcategory $\ST$ of $\SA$ put
\[\ST' = \add(\lbrace e_i^{\rho}(T) \vert ~ i\in Q_0, T\in\ST\rbrace).\]
Then $\ST'$ is an $(n+1)$-tilting subcategory of $\Rep(Q, \SA)$.
\end{stheorem}

\begin{proof}
We show the validity of the conditions of Definition \ref{HigherDef(BBOS)}. The validity of Condition $(i)$ follows by the similar argument in the proof of Theorem \ref{ProducingTauTiltingFromTiltingRep}.
For the validity of Condition $(ii)$, we have to show that for all $t\geq 1$ and $T_1, T_2\in\ST$,
\[\Ext^t_{\Rep(Q, A)}(e_i^{\rho}(T_1), e_j^{\rho}(T_2))=0.\]
But Proposition 5.2 of \cite{HJ} implies that
\[\Ext^t_{\Rep(Q, A)}(e_i^{\rho}(T_1), e_j^{\rho}(T_2))\cong\Ext^t_\SA(e_i^{\rho}(T_1)_j, T_2).\]
Now $e_i^{\rho}(T_1)_j$ is either zero or a sum of copies of $T_1$ and $\ST$ is an $n$-tilting subcategory, hence
\[\Ext^t_\SA(e_i^{\rho}(T_1)_j, T_2)=0.\]
Thus it follows that,\[\Ext^t_{\Rep(Q, A)}(e_i^{\rho}(T_1), e_j^{\rho}(T_2))=0.\]
For the Condition $(iii)$, we have to show that every object in $\ST'$ has projective dimension at most $n+1$. To this end, it is enough to show this fact for an additive generator $e_i^\rho(T)$ of $\ST'$, for some $i \in Q_0$ and some $T$ in $\ST$. By \cite[Lemma 3.5]{BBOS}, for every such generator, there exists a short exact sequence
\[0\lrt \bigoplus_{\alpha: r\rt t} e_t^\la(e_i^\rho(T)_r)\lrt \bigoplus_{r\in Q_0} e_r^\la(e_i^\rho(T)_r)\lrt e_i^\rho(T)\lrt 0.\]
Since for every $r, i\in Q_0$, $e_i^\rho(T)_r$ is zero or a sum of copies of $T$ and  projective dimension of $T$ is at most $n$, projective dimension of $e_i^\rho(T)_r$ is at most $n$. Now since  for every $t\in Q_0$, $e_t^\la$ preserves projective dimensions, projective dimension of $e_t^\la(e_i^\rho(T)_r)$  is at most $n$. Hence the above short exact sequence shows that projective dimension of $e_i^\rho(T)$ is at most $n+1$.

For the Condition $(iv)$,   we construct the  desired exact sequence for projectivs of  the  form $e_i^\la(P)$, where $P$ is a projective object in $\SA$. By the dual of Lemma 3.5 of \cite{BBOS}, there exists a short exact sequence
\[0\lrt e_i^\la(P)\lrt \bigoplus_{j\in Q_0}e_j^\rho(e_i^\la(P)_j)\lrt \bigoplus _{\alpha: t\rt j} e_t^\rho(e_i^\la(P)_j)\lrt 0.\]
First note that, since $\ST$ is an $n$-tilting subcategory of $\SA$, there exists an exact sequence
\begin{equation}\label{tilting}
0\lrt P\lrt T^0\lrt T^1\lrt \cdots\lrt T^n\lrt 0
\end{equation}
where $T^\ell\in\ST$, $\ell\in\lbrace 0, \cdots, n\rbrace$.

Since functors $e_i^\la$, $e_i$ and $e_i^\rho$ are all exact, the exact sequence (\ref{tilting}) induces the following two exact sequences
\[0 \lrt \bigoplus_{j\in Q_0}e_j^\rho(e_i^\la(P)_j) \lrt \bigoplus_{j\in Q_0}e_j^\rho(e_i^\la(T^0)_j) \lrt \cdots \lrt \bigoplus_{j\in Q_0}e_j^\rho(e_i^\la(T^n)_j) \lrt 0,\]
and
\[0\lrt \bigoplus _{\alpha: t\rt j} e_t^\rho(e_i^\la(P)_j) \lrt \bigoplus _{\alpha: t\rt j} e_t^\rho(e_i^\la(T^0)_j) \lrt \cdots \lrt \bigoplus _{\alpha: t\rt j} e_t^\rho(e_i^\la(T^n)_j) \lrt 0. \]
So we get the following diagram
\[\xymatrix{&&0\ar[d]&0\ar[d]&\\
0\ar[r] &e_i^\la(P)\ar[r] &\bigoplus_{j\in Q_0}e_j^\rho(e_i^\la(P)_j)\ar[r]^{\theta} \ar[d]&\bigoplus _{\alpha: t\rt j} e_t^\rho(e_i^\la(P)_j)\ar[r]\ar[d]&0\\
&&\bigoplus_{j\in Q_0}e_j^\rho(e_i^\la(T^0)_j)\ar[d]\ar@{..>}[r] &\bigoplus _{\alpha: t\rt j} e_t^\rho(e_i^\la(T^0)_j)\ar[d]& \\
&&\vdots\ar[d]&\vdots\ar[d]&\\
&& \bigoplus_{j\in Q_0}e_j^\rho(e_i^\la(T^n)_j)\ar@{..>}[r]\ar[d]&\bigoplus _{\alpha: t\rt j} e_t^\rho(e_i^\la(T^n)_j)\ar[d]&\\
&&0&0&
}\]

Since for every $r\geq 1$, $\Ext^r_\SA(\ST, \ST)=0$, and adjoint properties of the adjoint pairs $(e_i^{\la}, e_i)$ and $(e_i, e_i^{\rho})$ extends to $\Ext^1$, for every $\ell\in\lbrace 0, \cdots, n\rbrace$, $\bigoplus _{\alpha: t\rt j} e_t^\rho(e_i^\la(T^\ell)_j)$ is a relative injective object with respects to $\bigoplus_{j\in Q_0}e_j^\rho(e_i^\la(T^\ell)_j)$. Thus, we can construct the dotted maps starting from $\theta$.

Now by considering the mapping cone of the above diagram and applying a simple diagram chasing, we get the long exact sequence
\begin{align*}
0\lrt \ & e_i^\la(P)\lrt \bigoplus_{j\in Q_0}e_j^\rho(e_i^\la(T^0)_j) \lrt \bigoplus_{j\in Q_0}e_j^\rho(e_i^\la(T^1)_j)\oplus \bigoplus _{\alpha: t\rt j} e_t^\rho(e_i^\la(T^0)_j)\\  \lrt &\ \  \cdots \ \lrt \bigoplus_{j\in Q_0}e_j^\rho(e_i^\la(T^n)_j)\oplus\bigoplus _{\alpha: t\rt j} e_t^\rho(e_i^\la(T^{n-1})_j)\lrt \bigoplus _{\alpha: t\rt j} e_t^\rho(e_i^\la(T^n)_j)\lrt 0,
\end{align*}
such that all terms except $e_i^\la(P)$ are in $\ST'$. This is the desired exact sequence. Hence the proof is complete.
\end{proof}

\section*{Acknowledgments}
The authors would like to thank Professor Bernhard Keller for pointing out Lemma~\ref{Keller} and Lidia Angeleri-H\"ugel for the helpful discussions.
This work was partly done during a visit of the first author to the Institut des Hautes \'{E}tudes Scientifiques (IHES), Paris, France.
The first and third would like to thank the support and excellent atmosphere at IHES.
The work of the first and the second author is based upon research funded by Iran National Science Foundation (INSF) under project No. 4001480.
The third author is supported by the European Union’s Horizon 2020 research and innovation programme under the Marie Sklodowska-Curie grant agreement No 893654.

\end{document}